\documentclass[12pt,a4paper]{amsart}
\usepackage{fullpage,color}
\usepackage[mac]{inputenc}
\usepackage{amsmath,amsfonts,amssymb,amsthm,amscd}
\usepackage{graphicx}
\newtheorem{theorem}{Theorem}[section]

\newtheorem{lemma}[theorem]{Lemma}
\newtheorem{corollary}[theorem]{Corollary}

\newcommand{\C}{\mathbb C}

\newcommand{\Z}{\mathbb Z}

\DeclareMathOperator{\ord}{ord\ }

\usepackage{hyperref}

\title{On Mahler's transcendence measure for $e$}

\author{Anne-Maria Ernvall-Hyt\"onen}
\address{Matematik och Statistik, {\AA}bo Akademi University, Domkyrkotorget 1, 20500 {\AA}bo, Finland}
\email{anne-maria.ernvall-hytonen@abo.fi}

\author{Tapani Matala-aho}
\author{Louna Sepp\"al\"a}
\address{Matematiikka, PL 8000, 90014 Oulun yliopisto, Finland}
\email{tapani.matala-aho@oulu.fi}
\email{louna.seppala@student.oulu.fi}

\date{\today}

\thanks{The work of the author L.S.  was supported by the Magnus Ehrnrooth Foundation. 
The work of the author A.-M. E.-H. was supported by the Academy of Finland project 303820, and by the Finnish Cultural Foundation. \\The published version of this article may be found at \url{https://doi.org/10.1007/s00365-018-9429-3}.}



\begin{document}

\begin{abstract}
We present a completely explicit transcendence measure for $e$. This is a continuation and an improvement to the works of Borel, Mahler and Hata on the topic. 
Furthermore, we also prove a transcendence measure for an arbitrary positive integer power of $e$. The results are based on Hermite-Pad\'e approximations 
and on careful analysis of common factors in the footsteps of Hata.
\end{abstract}

\maketitle

\section{Introduction}

Let $m, H \ge 1$ be given and define $\omega (m, H)$ as the infimum of the numbers $r > 0$ satisfying the estimate
\begin{equation}\label{ensikaava}
\left|\lambda_0+\lambda_1 e+\lambda_2 e^{2}+ \ldots +\lambda_me^{m}\right| > \frac{1}{H^r},
\end{equation}
for all $\overline{\lambda}=(\lambda_0, \ldots ,\lambda_m)^T \in \mathbb{Z}^{m+1}\setminus\{\overline{0}\}$ with $\max_{1\le i\le m} \{|\lambda_i|\} \le H$.
Then any function greater than or equal to $\omega (m,H)$ may be called a transcendence measure for $e$ (see \cite{HATA}).
The quest to obtain good transcendence measures for $e$ dates back to Borel \cite{borel}. He proved that $\omega (m,H)$ is smaller than $c \log \log H$ for some positive constant $c$ depending only on $m$. This was considerably improved by Popken \cite{popken1,popken2}, who showed that $\omega(m,H)<m+\frac{c}{\log\log H}$ for some positive constant $c$ depending on $m$. Soon afterwards, Mahler \cite{mahler32} was able to get the dependance on $m$ explicit:
\[
\omega(m,H)<m+\frac{c m^2\log(m+1)}{\log \log H}
\]
with $c$ an absolute positive constant. The price he had to pay was that he was only able to prove the validity of the result in some 
subset of the set consisting of $m,H\in \mathbb{Z}_{\ge 1}$ with $H\geq 3$, unlike the results by Borel and Popken. 
Finally, in 1991, Khassa and Srinivasan \cite{KS} proved that the constant can be chosen to be $98$ in the set 
$m,H \in \mathbb{Z}_{\ge 1}$ with $\log \log H\geq d(m+1)^{6m}$ for some absolute constant $d>e^{950}$. 
Soon after, in 1995, Hata \cite{HATA} proved that the constant can be chosen to be $1$ in the set of $m$ and $H$ with 
$\log H\geq \max \left\{(m!)^{3\log m}, e^{24} \right\}$. A broader view about questions concerning transcendence measures can be found for instance in the books of Fel'dman and Nesterenko \cite{FELDMAN}, and Baker \cite{Baker}.

In \cite{HATA} Hata introduced a striking observation of big common factors hiding
in the auxiliary numerical approximation forms. These numerical approximation forms are closely related to the classical
Hermite-Pad\'e approximations (simultaneous approximations of the second type) of the exponential function used already by Hermite. 
The impact of the common factors was utilized in an asymptotic manner resulting in Theorem 1.2 in \cite{HATA}. 
Hata's Theorem 1.2 is sharper than Theorem 1.1 in his paper, but it is only valid for $H$ in an asymptotic sense: no explicit lower bound is given, instead, the theorem is formulated for a large enough $H$.

In this article we present a more extensive result, Theorem \ref{mainresult}. The improvements compared to Hata are made visible in its corollary, Theorem \ref{MAINCOROLLARY} below. Our Theorem \ref{MAINCOROLLARY} improves Hata's bound for the function $\omega$ in his Theorem 1.1, and extends the set of values of $H$ for which the result is valid whenever $m\geq 5$. In addition, this result makes Hata's Theorem 1.2 completely explicit, mainly due to our rigorous treatment of the common factors, giving rise to a more complicated behaviour visible in the term
\begin{equation}\label{kappavisible}
\kappa_m := \frac{1}{m} \sum_{\substack{p\le \frac{m+1}{2}\\ p\in \mathbb{P}}} 
\min_{0 \le j \le m} \left\{ \left\lfloor\frac{j}{p}\right\rfloor+\left\lfloor\frac{m-j}{p}\right\rfloor \right\} 
\frac{\log p}{p-1}\,w_p\!\left(s(m)e^{s(m)}\right),	
\end{equation}
where $w_n(x):= 1-\frac{n}{x}- \frac{n-1}{\log n}\frac{\log x}{x}$ for any $n \in \Z_{\ge 2}$, and $\mathbb{P}$ is the set of prime numbers.
We also give the exact asymptotic impact in \eqref{exact asynptotic}, as well as approximations for values of $\kappa_m$ for specific values of $m$.

\begin{theorem}\label{MAINCOROLLARY} Assume $\log H \ge s(m)e^{s(m)}$, where $s(2)=e$ and $s(m)=m(\log m)^2$ for $m \ge 3$. Now
\begin{equation}\label{mainestimate}
\omega (m, H) \le
\begin{cases}
2+ \frac{4.93}{\log \log H}, &m=2;\\
3+ \frac{6.49}{\log \log H}, &m=3;\\
4+ \frac{15.7}{\log \log H}, &m=4;\\
m +\left(1-\frac{2\kappa_m}{(\log m)^2}\right)\left(1-\frac{\kappa_m}{\log m}\right) \cdot \frac{m^2 \log m}{\log \log H}, & 5\leq m\leq 14;\\
m +\left(1-\frac{1+\kappa_m}{(\log m)^2}\right)\left(1-\frac{\kappa_m}{\log m}\right) \cdot \frac{m^2 \log m}{\log \log H}, &m \ge 15.
\end{cases}. 
\end{equation}
Asymptotically, we have
\begin{equation}\label{exact asynptotic}
\lim_{m\rightarrow \infty} \kappa_m =\kappa= \sum_{p\in\mathbb{P}} \frac{\log p}{p(p-1)}=0.75536661083\dots.
\end{equation}
\end{theorem}

Throughout his work, Hata assumed $\log H \ge \max \left\{ e^{s_1(m)}, e^{24} \right\}$ with the choice $s_1(m)\sim 3m(\log m)^2$. The bound $e^{24}$ is considerably larger than our bound $s(m)e^{s(m)}$ at its smallest (excluding the small cases $m=2,3,4$): $s(5)e^{s(5)}\approx e^{15.51\ldots}$. The choice of the function $s(m)$ was made as an attempt to balance between the amount of technical details, and the improvement of the function $\omega$ against the size of the set of the values of $H$.

In our Main theorem \ref{mainresult} we present a completely explicit transcendence measure for $e$, in terms of $m$ and $H$. 
The proof starts with Lemma \ref{lemma23}, which gives a suitable criterion for studying lower bounds of linear forms in given numbers. 
Furthermore, we exploit estimates for the exact inverse function $z(y)$ of the function 
$y(z) = z \log z$, $z \geq 1/e$, in the lines suggested  in \cite{HANCLETAL}. 
As an important consequence of using the function $z(y)$, the functional dependence in $H$ is improved compared to earlier considerations.

The method displayed in this paper is applicable to proving bounds of the type displayed in \eqref{ensikaava}. As an example, we will consider the case where the polynomial is sparse, namely, where several of the coefficients $\lambda_j$ in \eqref{ensikaava} are equal to zero. As a corollary of this, we derive a transcendence measure for positive integer powers of $e$.

It should be noted that all our results are actually valid over an imaginary quadratic field $\mathbb{I}$.

\section{Main result}

Let 
$z: \left[ -\frac{1}{e}, \infty \right[ \to \left[ \frac{1}{e}, \infty \right[$
denote the inverse function of the function 
$y(z) = z \log z$, $z \geq 1/e$, and denote further
\begin{equation*}
s(m)=m(\log m)^2, \quad m \ge 3; \quad s(2)=e;
\end{equation*}
\begin{equation*}
d=d(m)=\left(m+\frac{1}{2}\right)\log m-(1+ \kappa_m)m-1.02394, \quad m \ge 5;
\end{equation*}
\begin{equation*}
d(2)=0.3654, \quad d(3)=0.5139, \quad d(4)=1.6016;
\end{equation*}
\begin{equation*}
\begin{split}
B=B(m) = \;&m^2\log m-(1+ \kappa_m)m^2+(m+1)\log (m+1)+ \\
    &\frac{1}{2}m\log m-(1.02394+ \kappa_m)m+0.0000525, \quad m \ge 5;
\end{split}
\end{equation*}
\begin{equation*}
B(2)=2.4099, \quad B(3)=3.6433, \quad B(4)=9.7676;
\end{equation*}
\begin{equation*}
D=D(m)=(m+1)\log (m+1)-\kappa_m m+0.0000525+\frac{m}{e^{s(m)}}, \quad m \ge 5;
\end{equation*}
\begin{equation*}
D(2)=3.8111, \quad D(3)=5.1819, \quad D(4)=7.3631,
\end{equation*}
with $\kappa_m$ given in \eqref{kappavisible}.

Throughout this work, let $\mathbb{I}$ denote an imaginary quadratic field and
$\mathbb{Z}_{\mathbb{I}}$ its ring of integers.

\begin{theorem}\label{mainresult}
Let $m \ge 2$ and $\log H \ge s(m) e^{s(m)}$. With the above notations, the bound
\begin{equation}\label{mainreseq}
\left|\lambda_0+\lambda_1 e+ \ldots +\lambda_m e^m\right| > \frac{1}{2e^D} \left(2H\right)^{-m-\epsilon(H)},
\end{equation}
where
\begin{equation*}
\epsilon(H)\log 2H = B z\!\left(\frac{\log (2H)}{1 - \frac{d}{s(m)}}\right) + 
m \log\left(z\!\left(\frac{\log (2H)}{1 - \frac{d}{s(m)}}\right)\right),
\end{equation*}
holds for all 
$\,\overline{\lambda}=(\lambda_0, \lambda_1, \ldots ,\lambda_m)^T \in 
\mathbb{Z}_{\mathbb{I}}^{m+1}\setminus\{\overline{0}\}$ with $\max_{1\le j\le m} \{|\lambda_j|\} \le H$.
\end{theorem}

\begin{corollary}\label{corollaryofmainres}
With the assumptions of Theorem \ref{mainresult} we have
\begin{equation*}
\left|\lambda_0+\lambda_1e+ \ldots +\lambda_me^m\right| \ge 
\frac{1}{2 e^D} \left( \frac{s(m)-d}{s(m)+ \log (s(m))} \cdot \frac{\log\log(2H)}{\log(2H)}\right)^m \cdot (2H)^{-m -\frac{\widehat B}{\log \log (2H)}},
\end{equation*}
where
\begin{equation*}
\widehat B := \left(1+\frac{\log (s(m))}{s(m)} \right) \left(1 - \frac{d}{s(m)} \right)^{-1} \cdot B.
\end{equation*}
\end{corollary}

\section{Preliminaries, lemmas and notation}

Fix now $\Theta_1, \ldots ,\Theta_m\in\mathbb{C} \setminus \{0\}$.
Assume that we have a sequence of simultaneous linear forms
\begin{equation}\label{LINFORMS}
L_{k,j}(n)=B_{k,0}(n)\Theta_j+B_{k,j}(n),
\end{equation}
$k=0,1, \ldots, m, \: j=1, \ldots ,m,$
where the coefficients
\begin{equation*}
B_{k,j}=B_{k,j}(n)\in\mathbb{Z}_{\mathbb{I}},\quad k,j=0,1, \ldots, m,
\end{equation*}
satisfy the determinant condition
\begin{equation}\label{DET}
\Delta:=
\begin{vmatrix}
B_{0,0}  &B_{0,1} & \ldots & B_{0,m}\\
B_{1,0}  &B_{1,1} & \ldots & B_{1,m}\\
\vdots  & \vdots &  \ddots & \vdots \\
B_{m,0}  &B_{m,1} & \ldots & B_{m,m}\\
\end{vmatrix}
\ne 0.
\end{equation}
Further, let 
$
a,b,c,d\in\mathbb{R},\: a,c>0,
$
and suppose that
\begin{equation*}
|B_{k,0}(n)|\le Q(n)=e^{q(n)},
\end{equation*}
\begin{equation*}
\sum_{j=1}^{m}|L_{k,j}(n)|\le R(n)= e^{-r(n)},
\end{equation*}
where
\begin{equation}\label{qn}
q(n)=an\log n +bn,
\end{equation}
\begin{equation}\label{-rn}
-r(n)=-cn\log n +dn
\end{equation}
for all $k \in \{0,1, \ldots ,m \}$.
Let the above assumptions be valid for all $n\ge n_0$.

Before presenting a criterion for lower bound, Lemma \ref{lemma23}, we need some properties of the inverse function of the function
$y(z) = z \log z$, $z \geq 1/e$, considered in \cite{HANCLETAL}.
\begin{lemma}\label{inverse}
\cite{HANCLETAL} The inverse function $z(y)$ of the function
$y(z)= z \log z$, $z \geq 1/e$,
is strictly increasing. 
Define $z_0(y)=y$ and $z_n(y)=\frac{y}{\log z_{n-1} (y)}$ for $n\in\mathbb Z^+$. 
Suppose $y>e$, then 
$
z_1<z_3<\cdots <z<\cdots <z_2<z_0. 
$
Thus the inverse function may be given by the infinite nested logarithm fraction
\begin{equation*}
z(y) = \lim_{n\to\infty} z_{n}(y)=\frac{y}{\log \frac{y}{\log \frac{y}{\log \cdots}}},\quad y>e. 
\end{equation*}
\end{lemma}
Further, we denote
\begin{multline}\label{BCD}
B:=b+\frac{ad}{c},\ C:=a,\ D:=a+b+ae^{-s(m)},\ F^{-1}:= 2e^{D},\\
v := c - \frac{d}{s(m)},\ n_1:=\max\left\{n_0, e, e^{s(m)} \right\}.
\end{multline}

\begin{lemma}\label{lemma23}
Let $m\ge 1$ and $\log (2H) \ge v n_1 \log n_1$. Then, under the above assumptions \eqref{DET}-\eqref{-rn}, the bound
\begin{equation*}
\left|\lambda_0+\lambda_1\Theta_1+ \ldots +\lambda_m\Theta_m\right| > 
F \left(2H\right)^{-\frac{a}{c}-\epsilon(H)},
\end{equation*}
\begin{equation*}
\epsilon(H)\log 2H = B z\!\left(\frac{\log (2H)}{v}\right) + C\log\left(z\!\left(\frac{\log (2H)}{v}\right)\right)
\end{equation*}
holds for all 
$\,\overline{\lambda}=(\lambda_0, \lambda_1, \ldots ,\lambda_m)^T \in 
\mathbb{Z}_{\mathbb{I}}^{m+1}\setminus\{\overline{0}\}$
with $\max_{1 \le j \le m} \{ |\lambda_j |\} \le H$. 
\end{lemma}

\begin{proof}
We use the notation
\begin{equation*}
\Lambda:=\lambda_0+\lambda_1\Theta_1+ \ldots +\lambda_m\Theta_m,\quad \lambda_j\in\mathbb{Z}_{\mathbb{I}},
\end{equation*}
for the linear form to be estimated.
Using our simultaneous linear forms 
\begin{equation*}
L_{k,j}(n)=B_{k,0}(n)\Theta_j+B_{k,j}(n)
\end{equation*}
from \eqref{LINFORMS} we get
\begin{equation}\label{G+R}
B_{k,0}(n)\Lambda=W_k+\lambda_1 L_{k,1}(n)+ \ldots +\lambda_m L_{k,m}(n), 
\end{equation}
where
\begin{equation}\label{GKOK}
W_k(n)=B_{k,0}(n)\lambda_0- \lambda_1 B_{k,1}(n)- \ldots -\lambda_m B_{k,m}(n)\in\mathbb{Z}_{\mathbb{I}}.
\end{equation}
If now $W_k(n)\ne 0$, then by \eqref{G+R} and \eqref{GKOK} we get
\begin{equation*}
\begin{split}
1 \le |W_k(n)| &=|B_{k,0}(n)\Lambda-(\lambda_1 L_{k,1}+ \ldots +\lambda_m L_{k,m})| \\
&\le |B_{k,0}(n)||\Lambda |+\sum_{j=1}^{m}|\lambda_j||L_{k,j}(n)| \le Q(n)|\Lambda |+HR(n).
\end{split}
\end{equation*}
Now we take the largest $n_2$ with
\begin{equation}\label{takehatn}
n_2\ge n_1:=\max\left\{n_0, e, e^{s(m)} \right\}
\end{equation}
such that $\frac{1}{2}\le HR(n_2) $ with big enough $H$ (to be determined later). Consequently $HR(n_2+1)<\frac{1}{2}$.

According to the non-vanishing of the determinant \eqref{DET} and the assumption
$\overline{\lambda} \ne \overline{0}$,
it follows that $W_k(n_2+1)\ne 0$ for some integer $k\in[0,m]$. Hence we get the estimate
\begin{equation}\label{KIIKKU1} 
1< 2|\Lambda|Q(n_2+1) 
\end{equation}
for our linear form $\Lambda$, where we need to write $Q(n_2+1)$ in terms of $2H$.

Since $\frac{1}{2}\leq HR(n_2)$, we have
\begin{equation}\label{log2mHrhat}
\log (2H)\ge r(n_2) =c n_2 \log n_2 -dn_2 = n_2 \log n_2 \left(c - \frac{d}{\log n_2} \right).
\end{equation}
By \eqref{takehatn} we have $\log n_2\ge s(m)$. Thus
\begin{equation*}
\log (2H) \ge \left(c - \frac{d}{s(m)} \right) n_2 \log n_2 = v n_2 \log n_2,
\end{equation*}
or equivalently $n_2 \log n_2 \le \frac{\log (2H)}{v}$. Further, $n_2 \ge n_1 \ge e^{s(m)}$ by \eqref{takehatn}, which implies 
\begin{equation*}\label{lograja}
\frac{\log (2H)}{v}\geq n_2\log n_2 \geq  s(m) e^{s(m)}.
\end{equation*}

Then, by the properties of the function $z(y)$ given in Lemma \ref{inverse}, we get
\begin{equation}\label{zetan}
n_2 \le z\left(\frac{\log (2H)}{v}\right).
\end{equation}

Now we are ready to estimate $Q(n_2+1)=e^{q(n_2+1)}$ as follows:
\begin{equation}\label{intom}
\begin{split}
q(n_2+1) = \;&a(n_2+1)\log(n_2+1) +b(n_2+1)\\
< \;&a(n_2+1)\left(\log n_2+\frac{1}{n_2}\right)+b(n_2+1)\\
= \;&an_2\log n_2+a\log n_2+bn_2+a+b+\frac{a}{n_2}.
\end{split}
\end{equation}
By \eqref{log2mHrhat} we get
\begin{equation}\label{sub1}
n_2\log n_2\le \frac{1}{c}\left(\log (2H)+dn_2 \right).
\end{equation}
Substituting \eqref{sub1} into \eqref{intom} gives
\begin{equation*}
\begin{split}
q(n_2+1) &\leq \frac{a}{c}\left(\log (2H)+dn_2\right)+a\log n_2+bn_2+a+b+\frac{a}{n_2}\\
&\le \frac{a}{c}\log (2H)+\left(b+\frac{ad}{c}\right)n_2+ a \log n_2 +a+b+ae^{-s(m)},
\end{split}
\end{equation*}
where we applied \eqref{takehatn}.

Hence
\begin{equation*}
\begin{split}
Q(n_2+1) &\le \exp\left(\frac{a}{c}\log (2H)+\left(b+\frac{ad}{c}\right)n_2+ a \log n_2 +a+b+ae^{-s(m)}\right)\\
            &= \left(2H\right)^{\frac{a}{c}}e^{Bn_2+C\log n_2+D},
\end{split}
\end{equation*}
where $B$, $C$ and $D$ are precisely as in the formulation of Lemma \ref{lemma23}. The claim now follows from \eqref{KIIKKU1} and \eqref{zetan}.
\end{proof}

Let us now formulate a lemma that can be used to bound the function $z$. It is extremely useful while comparing our results with the results of others.

\begin{lemma}\label{lemma24}
If $y\geq s(m) e^{s(m)}$, we have $z(y) \geq e^{s(m)}$. When in addition $s(m) \ge e$, for the inverse function of $z(y)$ of the function $y(z) = z \log z$ it holds
\begin{equation*}
z(y) \le\left(1+\frac{\log(s(m))}{s(m)}\right) \frac{y}{\log y}.
\end{equation*}
\end{lemma}

\begin{proof}
Denote $z:= z(y)$ with $y \ge s(m) e^{s(m)}$. Then
\[
z=\frac{y}{\log z}
= \frac{y}{\log y}\frac{\log y}{\log z}
=\frac{y}{\log y}\left(1+\frac{\log\log z}{\log z}\right)
\le\frac{y}{\log y}\left(1+\frac{\log(s(m))}{s(m)}\right),
\]
because $\log z \ge s(m) \ge e$.
\end{proof}
\begin{corollary}\label{corollaari24}
If $c \le 1 + \frac{d}{s(m)}$, $\log (2H) \ge v n_1 \log n_1$ and $u:= 1+\frac{\log(s(m))}{s(m)}$, then
\begin{equation*}
\left|\lambda_0+\lambda_1\Theta_1+ \ldots +\lambda_m\Theta_m\right| \ge \frac{v^C}{2 e^{D}u^{C}} \left(\frac{\log\log(2H)}{\log(2H)}\right)^{C} \cdot (2H)^{-\frac{a}{c} -\frac{Bu}{v \log\log(2H)}}.
\end{equation*}
\end{corollary}
\begin{proof}
Since $c \le 1 + \frac{d}{s(m)}$, Lemma \ref{lemma24} gives
\begin{equation}\label{nhattuylaraja}
z\left(\frac{\log (2H)}{c - \frac{d}{s(m)}}\right) 
\le \left(1+\frac{\log(s(m))}{s(m)}\right) \frac{\frac{\log (2H)}{c - \frac{d}{s(m)}}}{\log \left(\frac{\log (2H)}{c - \frac{d}{s(m)}}\right)}
\le \frac{u}{v} \cdot \frac{\log (2H)}{\log \log (2H)},
\end{equation}
when we denote $u := 1+\frac{\log(s(m))}{s(m)}$. Lemma \ref{lemma23} now implies
\begin{equation*}
\begin{split}
\left|\lambda_0+\lambda_1\Theta_1+ \ldots +\lambda_m\Theta_m\right| &> \frac{1}{2e^D} \left(2H\right)^{-\frac{a}{c}} e^{-Bz\left(\frac{\log (2H)}{v}\right)-C\log \left(z\left(\frac{\log (2H)}{v}\right) \right)}\\
&\geq \frac{1}{2e^D} \left(2H\right)^{-\frac{a}{c}} e^{- \frac{u B}{v} \frac{\log(2H)}{\log\log(2H)}} \left( \frac{u}{v} \cdot \frac{\log(2H)}{\log\log(2H)} \right)^{-C}\\
&= \frac{v^C}{2 e^{D}u^{C}} \left(\frac{\log\log(2H)}{\log(2H)}\right)^{C} (2H)^{-\frac{a}{c} -\frac{Bu}{v \log\log(2H)}}.
\end{split}
\end{equation*}
\end{proof}

\section{Hermite-Pad\'e approximants for the exponential function}\label{Hermite-Pade}

Hermite-Pad\'e approximants of the exponential function date back to Hermite's \cite{HERMITE} transcendence proof of $e$; see
also \cite{WALDSCHMIDT}.

\begin{lemma}\label{Omegajasigma}
Let $\beta_0=0$, $\overline\beta=(\beta_0, \beta_1, \ldots ,\beta_m)^T \in \C^{m+1}$, and
$\overline l=(l_0,l_1, \ldots, l_m)^T\in\mathbb{Z}_{\geq1}^{m+1}$ be given and define 
$\sigma_{i}=\sigma_{i}\!\left(\overline l,\overline\beta \right)$ by
\begin{equation*}
\Omega\!\left(w, \overline{\beta} \right)=\prod_{j=0}^{m}(\beta_j-w)^{l_j}=
\sum_{i=l_0}^{L}\sigma_{i}w^i,\quad L=l_0+ \ldots +l_m.
\end{equation*}
Then 
\begin{equation}\label{sigma_i}
\sigma_i=\sigma_{i}\!\left(\overline l,\overline\beta \right)=(-1)^i\sum_{l_0+i_1+ \ldots +i_m=i}
\binom{l_1}{i_1} \cdots \binom{l_m}{i_m} \cdot \beta_1^{l_1-i_1}\cdots\beta_m^{l_m-i_m}
\end{equation}
and
\begin{equation*}
\sum_{i=0}^{L}\sigma_{i}i^{k_j}\beta_j^i=\sum_{i=l_0}^{L}\sigma_{i}i^{k_j}\beta_j^i=0
\end{equation*}
for all $j\in\{0,1, \ldots ,m\}$ and $k_j\in\{0, \ldots ,l_j-1\}$.
\end{lemma}

\begin{theorem}\label{Padeapprox}
Let $\alpha_0, \alpha_1, \ldots, \alpha_m$ be $m+1$ distinct complex numbers. Denote 
$\overline{\alpha} = (\alpha_0, \alpha_1, \ldots, \alpha_m)^T \in \C^{m+1}$ and 
$\overline{l} = ( l_0,  l_1, \ldots,  l_m)^T \in \Z_{\ge 1}^{m+1}$. Put
\begin{equation}\label{ALmu0}
A_{\overline{l},0}(t,\overline\alpha)= \sum_{i=l_0}^{L}t^{L-i} i! \sigma_i\!\left(\overline{l}, \overline{\alpha}\right).
\end{equation}
Then there exist polynomials $A_{\overline{l},j}(t,\overline\alpha)$ and remainders $R_{\overline{l},j}(t,\overline\alpha)$ such that
\begin{equation}\label{AF-A}
e^{\alpha_j t} A_{\overline{l},0}(t,\overline\alpha) - A_{\overline{l},j}(t,\overline\alpha) = R_{\overline{l},j}(t,\overline\alpha),
\end{equation}
where
\begin{equation*}
\begin{cases}
\deg_t A_{\overline{l},0}(t,\overline\alpha) = L-l_0,\\
\deg_t A_{\overline{l},j}(t,\overline\alpha) = L-l_j,\\
\underset{t=0}{\ord} R_{\overline{l},j}(t,\overline\alpha)\ge L+1
\end{cases}
\end{equation*}
for $j=1, \ldots, m$.
\end{theorem}
\begin{proof}
First we have
\begin{equation*}
A_{\overline{l},0}(t,\overline\alpha)= \sum_{i=l_0}^{L}t^{L-i} i! \sigma_i\!\left(\overline{l}, \overline{\alpha}\right) =\sum_{i=0}^L t^{L-i}i! \sigma_i\!\left(\overline{l}, \overline{\alpha}\right)= t^{L+1}\sum_{i=0}^L\frac{i!\sigma_i\!\left(\overline{l}, \overline{\alpha}\right)}{t^{i+1}},
\end{equation*}
since $\sigma_i\!\left(\overline{l}, \overline{\alpha}\right)=0$ for $0\leq i<l_0$. Using Laplace transform, we can write this as
\begin{equation}\label{LaplaceA0}
t^{L+1}\sum_{i=0}^L\frac{i!\sigma_i\!\left(\overline{l}, \overline{\alpha}\right)}{t^{i+1}} =t^{L+1}\sum_{i=0}^L\mathcal{L}\left(\sigma_i\!\left(\overline{l}, \overline{\alpha}\right) x^i\right)(t)
=t^{L+1}\int_0^{\infty}e^{-xt}\Omega(x,\overline\alpha)\mathrm{d}x.
\end{equation}
Then
\begin{equation*}
\begin{split}
e^{\alpha t}A_{\overline{l},0}(t,\overline\alpha) &= t^{L+1}\int_0^{\infty}e^{(\alpha-x)t}\Omega(x,\overline\alpha)\mathrm{d}x\\
&= t^{L+1} \int_{0}^{\infty}e^{-yt}\Omega(y+\alpha,\overline\alpha)\mathrm{d}y+t^{L+1} \int_{0}^{\alpha}e^{(\alpha-x)t}\Omega(x,\overline\alpha)\mathrm{d}x.
\end{split}
\end{equation*}
We have $\Omega(x,\overline\alpha)=\prod_{j=0}^{m}(\alpha_j-x)^{l_j}$ and consequently
\begin{equation}\label{Omegay+alpha}
\Omega(y+\alpha,\overline\alpha)=\prod_{j=0}^{m}(\alpha_j-\alpha-y)^{l_j} = \Omega\!\left( y, (\alpha_0-\alpha,\ldots ,\alpha_m-\alpha)^T \right).
\end{equation}
By setting $\alpha = \alpha_j, \, j= 1, \ldots, m$, we get the approximation formula
\begin{equation*}
e^{\alpha_j t} A_{\overline{l},0}(t,\overline\alpha) - A_{\overline{l},j}(t,\overline\alpha) = R_{\overline{l},j}(t,\overline\alpha),
\end{equation*}
where
\begin{equation}\label{Akjtalpha}
A_{\overline{l},j} (t,\overline\alpha)= A_{\overline{l},0}\!\left(t, (\alpha_0-\alpha_j,\ldots ,\alpha_m-\alpha_j)^T \right) = t^{L+1}\int_{0}^{\infty}e^{-yt}\Omega(y+\alpha_j,\overline\alpha)\mathrm{d}y
\end{equation}
and
\begin{equation*}
R_{\overline{l},j} (t,\overline{\alpha})= t^{L+1}\int_{0}^{\alpha_j}e^{(\alpha_j-x)t}\Omega(x,\overline\alpha)\mathrm{d}x, \quad j=1,\ldots ,m.
\end{equation*}

Going backwards in \eqref{LaplaceA0} with \eqref{Omegay+alpha} in mind we see that
\begin{equation*}
\begin{split}
A_{\overline{l},j}(t,\overline\alpha) &=t^{L+1}\int_{0}^{\infty}e^{-yt}\Omega(y+\alpha_j,\overline\alpha)\mathrm{d}y\\
&=t^{L+1}\sum_{i=0}^L\mathcal{L}\left(\sigma_i\!\left( \overline{l}, (\alpha_0-\alpha_j,\ldots ,\alpha_m-\alpha_j)^T \right) y^i \right)(t)\\
&= \sum_{i=l_j}^L t^{L-i}i! \sigma_i\!\left( \overline{l}, (\alpha_0-\alpha_j,\ldots ,\alpha_m-\alpha_j)^T \right).
\end{split}
\end{equation*}
Note that the coordinate $\alpha_j-\alpha_j=0$ corresponds to $\beta_0=0$ in Lemma \ref{Omegajasigma}, and consequently we now have $l_j$ in the place of $l_0$ in the definition of $\sigma_i$ \eqref{sigma_i}. Hence
$\sigma_i\!\left( \overline{l}, (\alpha_0-\alpha_j,\ldots ,\alpha_m-\alpha_j)^T \right) = 0$ for $0 \le i < l_j$, and $\deg_t A_{\overline{l},j}(t,\overline\alpha) = L-l_j$. In addition, ord$_{t=0}\, R_{\overline{l},j}(t,\overline\alpha)\ge L+1$ for 
$j=1,\ldots ,m$, since the function
\begin{equation*}
t \mapsto \int_{0}^{\alpha_j}e^{(\alpha_j-x)t}\Omega(x,\overline\alpha)\mathrm{d}x
\end{equation*}
is analytic at the origin.
\end{proof}

\begin{lemma}\label{intpol}
We have $\frac{1}{l_j!}A_{\overline{l},j}(t,\overline{\alpha})\in \mathbb{Z}[t,\alpha_1, \ldots ,\alpha_m]$ for all $j = 0, 1, \ldots, m$.
\end{lemma}

\begin{proof}
In the case $j=0$ we have
\begin{equation*}
\frac{1}{l_0!} A_{\overline{l},0}(t,\overline{\alpha})= 
\sum_{i=l_0}^{L}t^{L-i}\sigma_{i}\!\left(\overline l,\overline\alpha \right)\! \frac{i!}{l_0!}
\end{equation*}
by \eqref{ALmu0}. The claim clearly holds due to the definition of $\sigma_i$.

Next write
\begin{equation*}
A_{\overline{l},0} (t,\overline{\alpha}) e^{\alpha_j t} = \sum_{N=0}^\infty r_N t^N,
\end{equation*}
where
\begin{equation}\label{rN}
r_N = \sum_{N=h+n} \frac{\sigma_{L-h}\!\left(\overline l,\overline\alpha \right)\! (L-h)!}{n!} \alpha_j^n.
\end{equation}
By \eqref{AF-A} it is sufficient to show that
$\frac{r_N}{l_j!} \in \Z[\alpha_1, \ldots ,\alpha_m]$ for $N = 0, \ldots, L - l_j$. By \eqref{rN} we have
\begin{equation*}
\frac{1}{l_j!} r_N = \sum_{N=h+n} \sigma_{L-h}\!\left(\overline l,\overline\alpha \right)\! \frac{(L-h)!}{l_j! n!} \alpha_j^n,
\end{equation*}
where $h+n = N \le L - l_j$ implies $l_j + n \le L-h$, thus giving the result.
\end{proof}

\section{Determinant}\label{Determinant}

In order to fulfil the determinant condition \eqref{DET} we choose
\begin{equation}\label{l^k}
\overline{l}^{(k)} = (l, l, \ldots, l-1, \ldots, l)^T, \quad k=0, 1, \ldots, m,
\end{equation}
i.e. $l_i = l$ for $i=0, 1, \ldots, k-1, k+1, \ldots, m$, and $l_k=l-1$. Now $L=(m+1)l-1$.
Then we write
\begin{equation}\label{Bkjt}
\begin{cases}
A^*_{k,j}(t) := A_{\overline{l}^{(k)},j}(t,\overline{\alpha}),\quad j=0,1, \ldots ,m;\\
R^*_{k,j}(t) := R_{\overline{l}^{(k)},j}(t,\overline{\alpha}),\quad j=1, \ldots ,m,\\
\end{cases}
\end{equation}
for all $k=0,1, \ldots ,m$.

The non-vanishing of the determinant $\Delta$ follows from the next well-known lemma 
(see for example Mahler \cite[p. 232]{MAHLER2} or Waldschmidt \cite[p. 53]{WALDSCHMIDT}).

\begin{lemma}\label{Lemma5.1}
There exists a constant $c \neq 0$ such that
$$
\Delta (t) =
\begin{vmatrix}
A^*_{0,0}(t)  &A^*_{0,1}(t) & \ldots & A^*_{0,m}(t)\\
A^*_{1,0}(t)  &A^*_{1,1}(t) & \ldots & A^*_{1,m}(t)\\
\vdots  & \vdots &  \ddots & \vdots \\
A^*_{m,0}(t)  &A^*_{m,1}(t) & \ldots & A^*_{m,m}(t)\\
\end{vmatrix}
= ct^{m(m+1)l}.
$$
\end{lemma}

\begin{proof}
According to Theorem \ref{Padeapprox} and the equations in \eqref{Bkjt}, the degrees of the entries of the matrix defining $\Delta$ are
$$
\begin{pmatrix}
ml & ml-1 & \ldots & ml-1\\
ml-1 & ml & \ldots & ml-1\\
\vdots  & \vdots &  \ddots & \vdots \\
ml-1 & ml-1 & \ldots & ml\\
\end{pmatrix}_{\!(m+1) \times (m+1)}.
$$
We see that $\deg_t \Delta (t) = (m+1)ml$ and the leading coefficient $c$ is a product of the leading coefficients of 
$A^*_{0,0} (t), A^*_{1,1}(t), \ldots, A^*_{m,m}(t)$, which are non-zero.

On the other hand, column operations yield
$$
\Delta (t) =
\begin{vmatrix}
A^*_{0,0} (t) &-R^*_{0,1} (t) & \ldots & -R^*_{0,m} (t)  \\
A^*_{1,0} (t) &-R^*_{1,1} (t) & \ldots & -R^*_{1,m} (t)  \\
\vdots  & \vdots &  \ddots & \vdots \\
A^*_{m,0} (t) &-R^*_{m,1}(t) & \ldots & -R^*_{m,m} (t)   \\
\end{vmatrix},
$$
as $R^*_{k,j}(t)=e^{\alpha_j t} A^*_{k,0} (t) - A^*_{k,j}(t)$. 
By Theorem \ref{Padeapprox}, the order of each element in columns $1, \ldots, m$ is at least $L+1 = (m+1)l$. 
Therefore $\underset{t=0}{\ord} \Delta (t) \ge m(m+1)l$.
\end{proof}

\section{Common factors}\label{Commonfactors}

From now on we set $\alpha_j=j$ for $j=0,1, \ldots ,m$ and denote 
\begin{equation*}\label{}
B^*_{k, j} (t) := \frac{1}{(l-1)!} A^*_{k, j} (t)
\end{equation*}
for $j=0,1, \ldots, m$, $k=0, 1, \ldots, m$ and 
\begin{equation*}\label{}
L^*_{k, j} (t) := \frac{1}{(l-1)!} R^*_{k, j} (t)
\end{equation*}
for $j=1, \ldots, m$, $k=0, 1, \ldots, m$. 
Then, by Theorem \ref{Padeapprox}, we have a system of linear forms 
\begin{equation}\label{systemlinearforms}
B^*_{k,0} (t) e^{\alpha_j t}+ B^*_{k,j} (t) = L^*_{k,j} (t), \quad j=1, \ldots ,m;\ k=0, 1, \ldots, m,
\end{equation}
where
\begin{equation}\label{Bkjtl-1}
B^*_{k, j} (t) = \frac{t^{L+1}}{(l-1)!}\int_{0}^{\infty}e^{-yt} (0-j-y)^l (1-j-y)^l \cdots (k-j-y)^{l-1} \cdots (m-j-y)^l\mathrm{d}y
\end{equation}
for $j, k=0, 1, \ldots, m$, and
\begin{equation}\label{Lkjtl-1}
L^*_{k, j} (t) = \frac{t^{L+1}}{(l-1)!} \int_{0}^{j}e^{(j-x)t} (0-x)^l (1-x)^l \cdots (k-x)^{l-1} \cdots (m-x)^l \mathrm{d}x
\end{equation}
for $j=1, \ldots, m$, $k=0, 1, \ldots, m$.

Further, by Lemma \ref{intpol} holds $B^*_{k, j} (t)\in \mathbb{Z}[t]$ for all $j = 0, 1, \ldots, m$, $k=0, 1, \ldots, m$.
Next  we try to find a common factor from the integer coefficients of the new polynomials $B^*_{k,j} (t)$.

Let $m \in \Z_{\ge 1}$ in this section. We will also need the  $p$-adic valuation $v_p(n!)= \sum_{i=1}^{\infty} \left\lfloor \frac{n}{p^i}\right\rfloor$ and its well-known property
\begin{equation}\label{vpkertoma}
\frac{n}{p-1}-\frac{\log n}{\log p}-1\le v_p(n!)\le\frac{n-1}{p-1}
\end{equation}
(for reference, see \cite{KTT}).

\begin{theorem}\label{Casej=0}
For $k=0,1, \ldots, m$, we have
\[
\left( \prod_{\substack{p\le m\\ p\in \mathbb{P}}} p^{\left\lfloor \frac{m}{p}\right\rfloor v_p(l!) - v_p(l)} \right)^{-1} \cdot B^*_{k,0}(t) \in \Z[t].
\]
\end{theorem}
\begin{proof}
Let us start by writing the polynomial $B^*_{k,0}(t)$ from \eqref{Bkjtl-1} in a different way, using the representation \eqref{ALmu0}:
\begin{equation}\label{Bk0t}
\begin{split}
B^*_{k,0}(t) &= \frac{1}{(l-1)!} A_{\overline{l}^{(k)},0} \left(t,(0, 1, \ldots, m)^T \right) = 
\sum_{i=l_0}^{L}t^{L-i}\sigma_{i} \left( \overline{l}^{(k)}, (0, 1, \ldots, m)^T \right) \cdot \frac{i!}{(l-1)!} \\
&= \sum_{r=0}^{L-l_0}t^{r}\sigma_{L-r}\left( \overline{l}^{(k)}, (0, 1, \ldots, m)^T \right) \cdot \frac{(L-r)!}{(l-1)!},
\end{split}
\end{equation}
where $L=(m+1)l -1$ and, by \eqref{sigma_i},
\begin{equation*}
\begin{split}
\sigma_{L-r}\left( \overline{l}^{(k)}, (0, 1, \ldots, m)^T \right)= \;&(-1)^{L-r}\sum_{h_1+ \ldots +h_m=r} \frac{l_1!}{(l_1-h_1)!h_1!} \cdot \frac{l_2!}{(l_2-h_2)!h_2!} \cdots \\
&\frac{l_m!}{(l_m-h_m)!h_m!} 1^{h_1}2^{h_2}\cdots m^{h_m}.
\end{split}
\end{equation*}
So
\begin{equation*}
B^*_{k,0}(t) = \sum_{r=0}^{L-l_0} t^{r} (-1)^{L-r} \sum_{h_1+ \ldots +h_m=r} \frac{(L-r)!}{(l-1)! (l_1-h_1)! \cdots (l_m-h_m)!} 
\cdot \frac{l_1!}{h_1!}\cdots \frac{l_m!}{h_m!} \cdot 1^{h_1}\cdots m^{h_m}.
\end{equation*}
Here $\frac{(L-r)!}{(l-1)! (l_1-h_1)! \cdots (l_m-h_m)!} \in\mathbb{Z}$ because
\[
(l-1) + (l_1-h_1) + \ldots + (l_m-h_m) \le l_0 + l_1 + \ldots + l_m - (h_1 + \ldots h_m) = L-r.
\]
So, we may expect some common factors from the terms $\frac{l_1!}{h_1!}\cdots \frac{l_m!}{h_m!} \cdot 1^{h_1}2^{h_2}\cdots m^{h_m}$.

Let $p\leq m$ be a prime number. Now, using \eqref{vpkertoma},
\begin{equation}\label{vpterms}
\begin{aligned}
v_p& \left(\frac{l_1!}{h_1!} \cdot \frac{l_2!}{h_2!} \cdots \frac{l_m!}{h_m!} \cdot 1^{h_1}2^{h_2}\cdots m^{h_m}\right)
= \sum_{i=1}^m \left(v_p(l_i!) - v_p(h_i!) \right) + \sum_{\stackrel{i=1}{p | i}}^m h_i v_p(i)\\
\ge \; &\sum_{\stackrel{i=1}{p | i}}^m \left( v_p(l_i!)+h_i \left(v_p(i)-\frac{1}{p-1}\right) \right)\\
\ge \; &
\begin{cases}
\left( \left\lfloor \frac{m}{p}\right\rfloor -1 \right) v_p(l!) + v_p((l-1)!), &k \in \{ 1, \ldots, m \}; \\
\left\lfloor \frac{m}{p}\right\rfloor v_p(l!), &k=0
\end{cases} \\
\ge \; &\left( \left\lfloor \frac{m}{p}\right\rfloor -1 \right) v_p(l!) + v_p((l-1)!).
\end{aligned}
\end{equation}
Recall from \eqref{l^k} that $l_k = l-1$ while $l_j = l$ for $j \neq k$. Since $v_p(l!) = v_p(l) + v_p((l-1)!),$ the result \eqref{vpterms} can be written as
\begin{equation*}
v_p\left(\frac{l_1!}{h_1!} \cdot \frac{l_2!}{h_2!} \cdots \frac{l_m!}{h_m!} \cdot 1^{h_1}2^{h_2}\cdots m^{h_m}\right) \ge \left\lfloor \frac{m}{p}\right\rfloor v_p(l!) - v_p(l).
\end{equation*}
So, there is a factor
\begin{equation*}
p^{\left\lfloor \frac{m}{p}\right\rfloor v_p(l!) - v_p(l)} \left|\frac{l!}{h_1!}\cdots\frac{l!}{h_m!}\cdot 1^{h_1}2^{h_2}\cdots m^{h_m}\right.,
\end{equation*}
which is a common divisor of all the coefficients of $B^*_{k,0}(t)$. The proof is complete.\end{proof}
Now we need to find a common factor dividing all $B^*_{k,j}(t)$.

\begin{theorem}\label{commonfactorforBkj}
Assume $j\in\{1, \ldots ,m\}$.
Then there exists a positive integer
$$
D_{m,l} :=\prod_{\substack{p\le \frac{m+1}{2}\\ p\in \mathbb{P}}}p^{\nu_p}
$$
with
\[
\nu_p\geq  \left( \left\lfloor \frac{j}{p}\right\rfloor +\left\lfloor \frac{m-j}{p}\right\rfloor \right) v_p((l-1)!),
\]
satisfying
$$
D_{m,l}^{-1} \cdot B^*_{k,j}(t) \in \Z[t]
$$
for all $k=0,1, \ldots ,m$.
\end{theorem}
\begin{proof}
From our assumption $\alpha_i = i, \, i=1, \ldots, m$, and equations \eqref{Akjtalpha} and \eqref{Bk0t} it follows
\begin{equation*}
\begin{split}
B^*_{k,j}(t,\overline\alpha) &=\frac{1}{(l-1)!} A_{\overline{l}^{(k)},0}\left(t,(0-j, 1-j, \ldots ,m-j)^T \right) \\
&=\sum_{r=0}^{L-l_j}t^{r}\sigma_{L-r} \left(\overline{l}^{(k)}, (0-j, 1-j, \ldots ,m-j)^T \right) \frac{(L-r)!}{(l-1)!},
\end{split}
\end{equation*}
where
\begin{equation*}
\begin{split}
\sigma_{L-r}& \left(\overline{l}^{(k)}, (0-j, 1-j, \ldots ,m-j)^T \right) \\
= \;&(-1)^{L-r}\sum_{h_0+ \ldots + h_{j-1} + h_{j+1} \ldots + h_m=r} \frac{l_0!}{(l_0-h_0)!h_0!} \cdots \frac{l_{j-1}!}{(l_{j-1}-h_{j-1})!h_{j-1}!} \\
&\cdot \frac{l_{j+1}!}{(l_{j+1}-h_{j+1})!h_{j+1}!} \cdots \frac{l_m!}{(l_m-h_m)!h_m!}\\ &\cdot (0-j)^{h_0} (1-j)^{h_1}\cdots (-1)^{h_{j-1}}1^{h_{j+1}}\cdots(m-j)^{h_m}.
\end{split}
\end{equation*}

So
\begin{equation*}
\begin{aligned}
B^*_{k,j}(t) = &\sum_{r=0}^{L-l_j}t^{r}(-1)^{L-r} \sum_{h_0+ \ldots + h_{j-1} + h_{j+1} \ldots + h_m=r} \frac{(L-r)!}{(l-1)!} \\
&\cdot\frac{1}{(l_0-h_0)! \cdots (l_{j-1}-h_{j-1})! \cdot (l_{j+1}-h_{j+1})! \cdots (l_m-h_m)!} \\
&\cdot \frac{l_0!}{h_0!}\cdots \frac{l_{j-1}!}{h_{j-1}!} \cdot \frac{l_{j+1}!}{h_{j+1}!} \cdots \frac{l_m!}{h_m!} (0-j)^{h_0}(1-j)^{h_1}\cdots (-1)^{h_{j-1}}1^{h_{j+1}}\cdots(m-j)^{h_m}.
\end{aligned}
\end{equation*}
As before, we may expect some common factors from the terms
\[
T_j := \frac{l_0!}{h_0!}\cdots \frac{l_{j-1}!}{h_{j-1}!} \cdot \frac{l_{j+1}!}{h_{j+1}!} \cdots \frac{l_m!}{h_m!} 
\cdot (0-j)^{h_0}(1-j)^{h_1}\cdots (-1)^{h_{j-1}}1^{h_{j+1}}\cdots(m-j)^{h_m}.
\]

Let $p \le \frac{m+1}{2}$. With considerations similar to those in \eqref{vpterms}, we get
\begin{equation*}
\begin{split}
v_p (T_j) &\ge
\begin{cases}
\left( \left\lfloor \frac{j}{p} \right\rfloor - 1 \right) v_p (l!) + v_p ((l-1)!) + \left\lfloor \frac{m-j}{p} \right\rfloor v_p (l!), &k \in \{ 0, 1, \ldots j-1 \};\\
\left\lfloor \frac{j}{p} \right\rfloor v_p (l!) + \left( \left\lfloor \frac{m-j}{p} \right\rfloor -1 \right) v_p (l!) + v_p ((l-1)!), &k \in \{ j+1, \ldots m-j \};\\
\left\lfloor \frac{j}{p}\right\rfloor v_p(l!)+\left\lfloor \frac{m-j}{p}\right\rfloor v_p(l!), &k=j
\end{cases}\\
&\ge \left( \left\lfloor \frac{j}{p}\right\rfloor +\left\lfloor \frac{m-j}{p}\right\rfloor \right) v_p(l!) - v_p (l)\\
&\ge \left( \left\lfloor \frac{j}{p}\right\rfloor +\left\lfloor \frac{m-j}{p}\right\rfloor \right) v_p((l-1)!).
\end{split}
\end{equation*}
\end{proof}

Combining Theorems \ref{Casej=0} and \ref{commonfactorforBkj} gives us the complete result:

\begin{corollary}
For all $k, j = 0, 1, \ldots, m$ we have
$$
D_{m,l}^{-1} \cdot B^*_{k,j}(t) \in \Z[t].
$$ 
\end{corollary}

\begin{theorem}\label{kappatheorem}
Let $l\geq s(m)e^{s(m)}$. Then the common factor $D_{m,l}$ satisfies the bound
\begin{equation}\label{cfestimate}
D_{m,l}\geq e^{\kappa_m ml},
\end{equation}
where
\begin{equation}\label{kappaarvio}
\kappa_m := \frac{1}{m} \sum_{\substack{p\le \frac{m+1}{2}\\ p\in \mathbb{P}}} 
\min_{0 \le j \le m} \left\{ \left\lfloor\frac{j}{p}\right\rfloor+\left\lfloor\frac{m-j}{p}\right\rfloor \right\} 
\frac{\log p}{p-1}\,w_p\!\left(s(m)e^{s(m)}\right),	
\end{equation}
and  $w_p(x):= 1-\frac{p}{x}- \frac{p-1}{\log p}\frac{\log x}{x}$.  Further,
\begin{equation}\label{simplifynum}
\kappa_m\ge  w_{\frac{m+1}{2}}\!\left(s(m)e^{s(m)}\right) \frac{1}{m} \sum_{\substack{p\le \frac{m+1}{2}\\ p\in \mathbb{P}}} 
              \left(\left\lfloor\frac{m+1}{p}\right\rfloor-1\right) 
              \frac{\log p}{p-1},
\end{equation}
and asymptotically we have
\begin{equation}\label{kappaasymptotically}
\lim_{m \to \infty} \kappa_m = \kappa= \sum_{p\in\mathbb{P}} \frac{\log p}{p(p-1)}=0.75536661083\dots.
\end{equation}
\end{theorem}

\begin{proof} 
We begin with the estimate of Theorem \ref{commonfactorforBkj}:
\begin{equation*}
\nu_p \ge \left(\left\lfloor\frac{j}{p}\right\rfloor+\left\lfloor\frac{m-j}{p}\right\rfloor\right)v_p((l-1)!).
\end{equation*}
Then
\begin{equation*}
\begin{split}
\prod_{p\le \frac{m+1}{2}} p^{\nu_p} &\geq \prod_{p\le \frac{m+1}{2}} p^{\left(\left\lfloor\frac{j}{p}\right\rfloor+\left\lfloor\frac{m-j}{p}\right\rfloor\right)v_p((l-1)!)} \\
&= \exp \left( \sum_{p\le \frac{m+1}{2}} \left(\left\lfloor\frac{j}{p}\right\rfloor+\left\lfloor\frac{m-j}{p}\right\rfloor\right)v_p((l-1)!)\log p \right).
\end{split}
\end{equation*}
We use the estimate
\begin{equation*}
\left\lfloor\frac{j}{p}\right\rfloor+\left\lfloor\frac{m-j}{p}\right\rfloor \ge \min_{0 \le j \le m} \left\{ \left\lfloor\frac{j}{p}\right\rfloor+\left\lfloor\frac{m-j}{p}\right\rfloor \right\}
\end{equation*}
since we are estimating a common divisor of all $B^*_{k, j}$. 
Next we use the property \eqref{vpkertoma} and the assumption $l \ge s(m)e^{s(m)}$ in order to estimate $v_p((l-1)!)\log p$:
\begin{equation*}
\begin{split}
v_p((l-1)!)\log p &\ge \left( \frac{l-1}{p-1}-\frac{\log (l-1)}{\log p}-1 \right) \log p \\
&\ge l \frac{\log p}{p-1}\left( 1-\frac{p}{s(m)e^{s(m)}}- \frac{p-1}{\log p}\frac{\log \!\left(s(m)e^{s(m)}\right)}{s(m)e^{s(m)}} \right) .
\end{split}
\end{equation*}
Altogether $\prod_{p\le \frac{m+1}{2}} p^{\nu_p} \ge e^{\kappa_m ml},$ where
\begin{equation*}\label{}
\kappa_m := \frac{1}{m} \sum_{p\le \frac{m+1}{2}} 
\min_{0 \le j \le m} \left\{ \left\lfloor\frac{j}{p}\right\rfloor+\left\lfloor\frac{m-j}{p}\right\rfloor \right\} 
\frac{\log p}{p-1}\left( 1-\frac{p}{s(m)e^{s(m)}}- \frac{p-1}{\log p}\frac{\log \!\left(s(m)e^{s(m)}\right)}{s(m)e^{s(m)}} \right),
\end{equation*}
proving the estimate \eqref{cfestimate}.

Next we study the bound \eqref{simplifynum}.
Let $x\in\mathbb{R}_{>1}$ be fixed, then
\begin{equation}\label{wvahenee}
w_y(x)>w_{z}(x)
\end{equation}
when $2\leq y<z$. To prove \eqref{wvahenee} above we differentiate the function $w_y(x)$:
\begin{equation*}\label{}
\frac{\partial }{\partial y}w_y(x)=-\frac{1}{x}- \frac{1}{\log y}\frac{\log x}{x}+\frac{y-1}{y(\log y)^2}\frac{\log x}{x}
=-\frac{1}{x}-\frac{\log x}{x\log y}\left(1-\left(1-\frac{1}{y}\right)\frac{1}{\log y}\right)<0,
\end{equation*}
since $\log y>1-\frac{1}{y}$ when $y\geq 2$. Next write
\begin{equation}\label{jakoalg}
m+1=hp+\overline{m},\quad j=lp+\overline{j},\quad h,l, \overline{m}, \overline{j},\in\mathbb{Z}_{\ge 0},\quad 
0\le \overline{m}, \overline{j}\le p-1.
\end{equation}
Then
\begin{equation}\label{lattiasumma}
\begin{split}
\left\lfloor\frac{j}{p}\right\rfloor+\left\lfloor\frac{m-j}{p}\right\rfloor &=\left\lfloor l+\frac{\overline{j}}{p}\right\rfloor+\left\lfloor \frac{m-lp-\overline{j}}{p}\right\rfloor\\ 
&=l+\left\lfloor \frac{m+1}{p} -l-1+ \frac{p-1-\overline{j}}{p} \right\rfloor\geq \left\lfloor\frac{m+1}{p}\right\rfloor-1=h-1.
\end{split}											
\end{equation}
Thus, the bound $\min_{0 \le j \le m} \left\{ \left\lfloor\frac{j}{p}\right\rfloor+\left\lfloor\frac{m-j}{p}\right\rfloor \right\}
\geq   \left\lfloor\frac{m+1}{p}\right\rfloor-1
\geq   \frac{m}{p}-2$ together with \eqref{wvahenee} verifies the estimate
\begin{equation*}
\begin{split}
\kappa_m &\ge w_{\frac{m+1}{2}}\!\left(s(m)e^{s(m)}\right)\frac{1}{m} \sum_{p\le \frac{m+1}{2}} \left(\left\lfloor\frac{m+1}{p}\right\rfloor-1\right) \frac{\log p}{p-1}\\
& \ge w_{\frac{m+1}{2}}\!\left(s(m)e^{s(m)}\right) \sum_{p\le \frac{m+1}{2}} \left(1-\frac{2p}{m}\right) \frac{\log p}{p(p-1)}.	
\end{split}											
\end{equation*}
Hence, by restricting the sum to primes $p\le \sqrt{m}$, we get  
\begin{equation*}
\kappa_m \geq w_{\frac{m+1}{2}}\!\left(s(m)e^{s(m)}\right) \left(1-\frac{2}{\sqrt{m}}\right) \sum_{p\le \sqrt{m}} \frac{\log p}{p(p-1)} \stackrel{m \to \infty}{\to} \sum_{p\in\mathbb{P}} \frac{\log p}{p(p-1)}.
\end{equation*}
On the other hand,
\begin{equation*}
\begin{split}
\kappa_m &= \frac{1}{m} \sum_{p\le \frac{m+1}{2}} \min_{0 \le j \le m} \left\{ \left\lfloor\frac{j}{p}\right\rfloor+\left\lfloor\frac{m-j}{p}\right\rfloor \right\} \frac{\log p}{p-1} \, w_p\!\left(s(m)e^{s(m)}\right) \\
&\le \frac{1}{m} \sum_{p\le \frac{m+1}{2}} \left\lfloor\frac{m}{p}\right\rfloor \frac{\log p}{p-1}\le \sum_{p\le \frac{m+1}{2}}\frac{\log p}{p(p-1)} \stackrel{m \to \infty}{\to} \sum_{p\in\mathbb{P}} \frac{\log p}{p(p-1)}.
\end{split}
\end{equation*}
This proves the asymptotic behaviour \eqref{kappaasymptotically}. As for the numerical value in \eqref{kappaasymptotically}, see the sequence \href{https://oeis.org/A138312}{A138312} in \cite{sloane}.
\end{proof}

With $s(m)=m(\log m)^2$, for instance \eqref{kappaarvio} gives
\begin{equation*}
\kappa_m\geq
\begin{cases}
0, & m=2;\\
0.215544, & m=3;\\
0.173121, & m=4;\\
0.387118, & m=5;\\
0.322600, & m=6;\\
0.375535, & m=7;\\
0.397256, & m=8;\\
\end{cases} \quad \textrm{and}\quad \kappa_m\geq
\begin{cases}
0.474840, & m=9; \\
0.427356, & m=10;\\
0.501455, & m=11;\\
0.459667, & m=12;\\
0.502575, & m=13;\\
0.534653, & m=14.
\end{cases}
\end{equation*}

Note that to simplify numerical computations for large $m$, the estimate \eqref{simplifynum} is already rather sharp, where in addition the factor 
$w_{\frac{m+1}{2}}\!\left(s(m)e^{s(m)}\right)$ 
is very close to 1.

\begin{lemma}
It holds that $\kappa_m \ge 0.5$ for all $m \ge 13$.\end{lemma}

\begin{proof}
By \eqref{jakoalg} and \eqref{lattiasumma} we get
\begin{equation*}
\frac{1}{m}\left(\left\lfloor\frac{j}{p}\right\rfloor+\left\lfloor\frac{m-j}{p}\right\rfloor\right) \geq \frac{h-1}{m} \geq \frac{1}{p}\left(1-\frac{2p-2}{m}\right).
\end{equation*}
We choose, for example, $1-\frac{2p-2}{m} \geq \frac{9}{10}$ which is equivalent to $p\leq \frac{m}{20}+1$. Then
\begin{equation}\label{inclb}
\kappa_m \ge \frac{9}{10} w_{\frac{m+1}{2}}\!\left(s(m)e^{s(m)} \right) \sum_{p\le \frac{m}{20}+1} \frac{\log p}{p(p-1)}.	
\end{equation}
Now
$$
w_{\frac{m+1}{2}}\!\left(s(m)e^{s(m)} \right) = 1- \frac{\frac{m+1}{2} + \frac{ \frac{m+1}{2}-1 }{\log \left( \frac{m+1}{2} \right)} \cdot \log \left( m(\log m)^2 e^{m(\log m)^2} \right)}{m(\log m)^2 e^{m(\log m)^2}} > 1- 10^{-666},
$$
when $m \ge 80$. In \eqref{inclb} we have an increasing lower bound for $\kappa_m$, and therefore
\begin{equation*}
\kappa_m \ge \frac{9}{10} w_{\frac{80+1}{2}}\!\left(s(80)e^{s(80)}\right) \sum_{p\le \frac{80}{20}+1} \frac{\log p}{p(p-1)} > \frac{9}{10} \cdot \left( 1- 10^{-666} \right) \cdot \sum_{p\le 5} \frac{\log p}{p(p-1)}\ge 0.549133,
\end{equation*}
when $m \ge 80$. As for $13 \le m \le 79$, the estimate $\kappa_m \ge 0.5$ is quickly verified using Sage \cite{Sage} and estimate \eqref{simplifynum}.
\begin{center}
\includegraphics[width=1\textwidth]{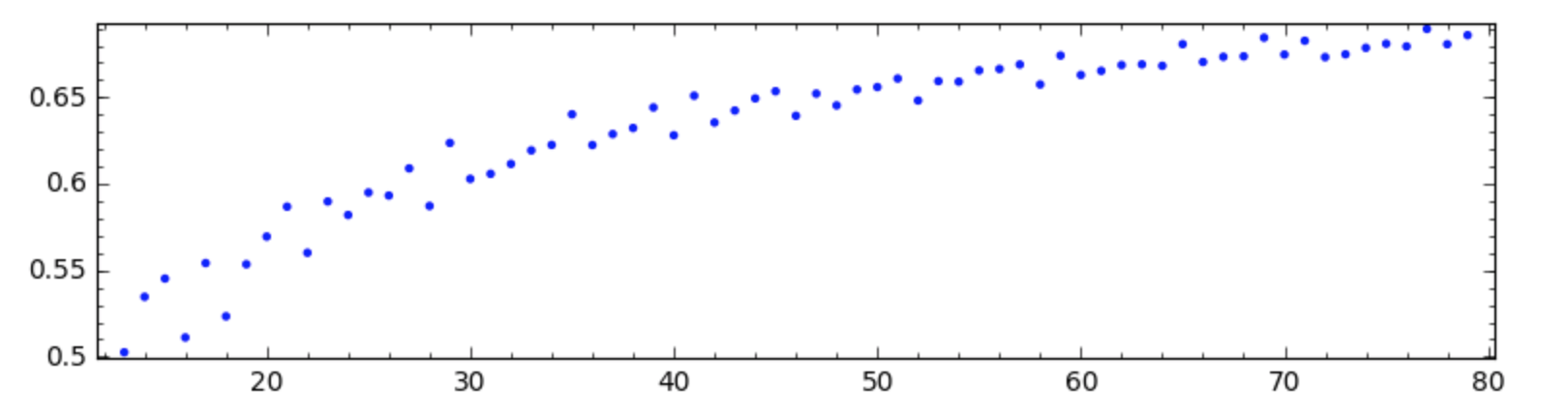}\end{center}
\end{proof}

\section{Numerical linear forms}

By extracting the common factor $D_{m,l}$ from the linear forms \eqref{systemlinearforms} we are led to the numerical linear forms 
\begin{equation}\label{numeerinenlineaari}
B_{k,0} e^{j}+ B_{k,j}  = L_{k,j}, \quad j=1, \ldots ,m;\ k=0,1, \ldots ,m,
\end{equation}
where
\begin{equation*}
B_{k,j}:=\frac{1}{D_{m,l}}B^*_{k,j}(1),\quad j=0,1, \ldots ,m;\ k=0,1, \ldots ,m,
\end{equation*}
are integers and
\begin{equation*}
L_{k,j}:=\frac{1}{D_{m,l}}L^*_{k,j}(1),\quad j=1, \ldots ,m;\ k=0,1, \ldots ,m,.
\end{equation*}
Note that $B_{k,j} = B_{k,j}(l)$ and $L_{k,j} = L_{k,j}(l)$.

According to \eqref{l^k}, now $L=(m+1)l-1$. We have $s(m) = m (\log m)^2$ for $m \ge 3$, $s(2)=e$. 
Because of the condition \eqref{takehatn} we have the assumption $l \ge e^{s(m)}$.

The following two  lemmas give the necessary estimates for the coefficients $B_{k,0}$ and the remainders $L_{k,j}$ of the linear forms 
\eqref{numeerinenlineaari}. In the subsequent estimates we shall use Stirling's formula (see e.g. \cite{stirling}, formula 6.1.38) in the form
\begin{equation*}
n! = \sqrt{2 \pi} n^{n+\frac{1}{2}} e^{-n + \frac{\theta (n)}{12n}}, \quad 0 < \theta (n) < 1.
\end{equation*}
Then
\begin{equation}\label{(l-1)fact}
\frac{1}{(l-1)!} \le \exp \left(-l\log l +l(\log l - \log(l-1))+l-1 +\frac{1}{2}\log(l-1)-\log\sqrt{2\pi} \right).
\end{equation}

\begin{lemma}\label{Estimates}
Let $l\geq e^{m (\log m)^2}$ when $m \ge 3$, and $l \ge e^e$ when $m=2$. We have
\begin{equation*}
|B_{k,0}| \le \exp \left(ml\log l+l((m+1)\log (m+1)-(1+\kappa_m)m+0.0000525) \right);
\end{equation*}
for $k=0,1, \ldots ,m$ and $m\geq 5$. When $m=2,3,4$, we have the bounds
\begin{equation}\label{smallvalpol}
\begin{cases}
|B_{k,0}|\leq \exp\left(2l\log l+1.6791l\right), & \textrm{when m=2};\\
|B_{k,0}|\leq \exp\left(3l\log l+2.1016l\right), & \textrm{when m=3};\\
|B_{k,0}|\leq \exp\left(4l\log l+3.3612l\right), & \textrm{when m=4}.
\end{cases}
\end{equation}
\end{lemma}

\begin{proof}
The structure of the proof is the following: First we treat the term $B^*_{k,0}(1)$ by using the formulas given for it to obtain a bound. 
Then we factor out the common factor $D_{m,l}$ of $B^*_{k,0}(1)$ yielding to $B_{k,0}$. 
Finally, the bound $Q(l)$ is then the bound for $B_{k,0}$.

By \eqref{Bkjtl-1} we have
\begin{equation*}
B^*_{k,0}(1)=\frac{1}{(l-1)!}\int_0^{\infty}e^{-x}\frac{\prod_{j=0}^m (j-x)^l}{k-x} \mathrm{d}x.
\end{equation*}
Let us split the integral into following pieces:
\begin{equation*}
\int_0^{\infty}e^{-x}\frac{\prod_{j=0}^m (j-x)^l}{k-x} \mathrm{d}x
= \left(\int_0^m+\int_m^{2(m+1)l}+\int_{2(m+1)l}^{\infty}\right)e^{-x}\frac{\prod_{j=0}^m (j-x)^l}{k-x} \mathrm{d}x.
\end{equation*}
When $x\geq m$,
\begin{equation*}
\left| \frac{\prod_{j=0}^m (j-x)^l}{k-x} \right| = x^l (x-1)^l \cdots (x-k)^{l-1} \cdots (x-m)^l \leq x^{(m+1)l -1} \le x^{(m+1)l}.
\end{equation*}
Hence, we may estimate
\[
\left|\left(\int_m^{2(m+1)l}+\int_{2(m+1)l}^{\infty}\right)e^{-x}\frac{\prod_{j=0}^m (j-x)^l}{k-x} \mathrm{d}x \right|
\leq \left(\int_m^{2(m+1)l}+\int_{2(m+1)l}^{\infty}\right)e^{-x}x^{(m+1)l} \mathrm{d}x.
\]
Write $f(x)=e^{-x}x^{(m+1)l}$. Now $f'(x)=(-x+(m+1)l)e^{-x}x^{(m+1)l-1}$, which has a unique zero at $x=(m+1)l$. The function $f(x)$ is increasing for $x\leq (m+1)l$ and decreasing for $x\geq (m+1)l$. Let us now estimate the integrals. The function $f(x)$ obtains its maximum at $x=(m+1)l$, and we may thus estimate
\begin{equation*}
\int_{m}^{2(m+1)l}e^{-x}x^{(m+1)l}\mathrm{d}x\leq 2(m+1)le^{-(m+1)l}((m+1)l)^{(m+1)l}.
\end{equation*}
On the interval $x\geq 2(m+1)l$, the function $f(x)$ is decreasing. Our aim is to find an upper bound for the integral using a geometric sum. Let us first write
\begin{equation*}
\int_{2(m+1)l}^{\infty}e^{-x}x^{(m+1)l} \mathrm{d}x\leq\sum_{h=0}^{\infty}e^{-2(m+1)l-h}(2(m+1)l+h)^{(m+1)l}.
\end{equation*}
Notice that $\frac{f(x+1)}{f(x)}=e^{-1}\left(1+\frac{1}{x}\right)^{(m+1)l}\leq e^{-1/2}$, when $x\geq 2(m+1)l$. It follows that $f(x+h) \le \left( e^{-\frac{1}{2}} \right)^h f(x)$. Hence
\begin{equation*}
\begin{split}
\sum_{h=0}^{\infty}e^{-2(m+1)l-h}(2(m+1)l+h)^{(m+1)l}
&\leq \frac{e^{-2(m+1)l}(2(m+1)l)^{(m+1)l}}{(1-e^{-1/2})} \\
&<2.55 e^{-2(m+1)l}(2(m+1)l)^{(m+1)l}.
\end{split}
\end{equation*}

Finally, we have to estimate the first integral. We have
\begin{equation*}
\max_{n\leq x\leq n+1}\prod_{j=0}^m |j-x|^{l-1} \leq \max_{0\leq x\leq 1}\prod_{j=0}^m |j-x|^{l-1}
\end{equation*}
for $0\leq n\leq m-1$. Now
\begin{equation*}
\left| \frac{\prod_{j=0}^m (j-x)^l}{k-x} \right|
\le m! \max_{0\leq x\leq 1}\prod_{j=0}^m |j-x|^{l-1}.
\end{equation*}
Hence
\begin{equation*}
\left| \int_0^m e^{-x}\frac{\prod_{j=0}^m (j-x)^l}{k-x} \mathrm{d}x \right|
\leq  \frac{m!(m!)^{l-1}}{(5!)^{l-1}}\int_0^m e^{-x}\max_{0\leq x\leq 1}\prod_{j=0}^5|j-x|^{l-1} \mathrm{d}x 
\leq \;\frac{(m!)^l16.91^{l-1}}{120^{l-1}}
\end{equation*}
when $m\geq 5$. When $m<5$, we can estimate
\[
\left| \int_0^m e^{-x}\frac{\prod_{j=0}^m (j-x)^l}{k-x} \mathrm{d}x \right| \le
\begin{cases}
\frac{2^{l+1}}{3^{3(l-1)/2}}, & \textrm{when }m=2;\\
6, & \textrm{when }m=3;\\
24\cdot  3.632^{l-1}, &\textrm{when }m=4.
\end{cases}
\]

We may conclude that
\begin{equation*}
\begin{split}
|B^*_{k,0}(1)|\leq \;&\frac{(m!)^l16.91^{l-1}}{(l-1)!120^{l-1}}+\frac{2(m+1)l((m+1)l)^{(m+1)l}}{(l-1)!e^{(m+1)l}}+\frac{2.55e^{-2(m+1)l}(2(m+1)l)^{(m+1)l}}{(l-1)!}\\
\leq \;&\frac{6(m+1)l}{(l-1)!}e^{-(m+1)l}((m+1)l)^{(m+1)l}\\ 
\leq 
\;&\exp\bigg(m l \log l+l \left((m+1)\log (m+1)-m +\log l-  \log (l-1) \right)\\
&+ \log l +\frac{1}{2}\log (l-1)+\log (m+1)+\log 6-1-\log\sqrt{2\pi}\bigg).
\end{split}
\end{equation*}

Next we take into account the common factor $D_{m,l}$ estimated by $e^{\kappa_m ml}$. 
Remember that $B_{k,0}$ will be the expression that is obtained when $B^*_{k,0}(1)$ is divided by the common factor. Now
\begin{equation}\label{Bk0estimate2}
\begin{split}
|B_{k,0}| \leq \exp\bigg(&m l \log l+l \left((m+1)\log (m+1) -(1+\kappa_m)m +\log l-  \log (l-1) \right)\\
& +\log l+\frac{1}{2}\log (l-1) +\log (m+1)+\log 6-1-\log\sqrt{2\pi}\bigg).
\end{split}
\end{equation}
Since $m \ge 5$ and $l \ge e^{m (\log m)^2} \ge e^{5 (\log 5 )^2}$, we have
\begin{equation}\label{est1}
\log l-  \log (l-1) \le 0.000002373
\end{equation}
and
\begin{equation}\label{est2}
 \frac{\log l}{l} +\frac{\log (l-1)}{2l} + \frac{\log (m+1)}{l} + \frac{\log 6-1-\log\sqrt{2\pi}}{l} \le 0.00005005.
\end{equation}
At last, estimate \eqref{Bk0estimate2} with \eqref{est1} and \eqref{est2} yields
\begin{equation*}
|B_{k,0}| \leq \exp\left(ml\log l+l((m+1)\log (m+1)-(1+\kappa_m)m+0.0000525)\right).
\end{equation*}

When $m=2$, we have $|B_{k,0}|=|B_{k,0}^*(1)|$ and $l\geq \left\lceil e^e \right\rceil =16$. Now
\begin{multline*}
|B_{k,0}| \le \frac{1}{(l-1)!} \left( 2^{l+1}3^{-3(l-1)/2}+2\cdot 3\cdot l e^{-3l}(3l)^{3l}+2.55\cdot e^{-6l}(6l)^{3l} \right)\\
\leq \frac{3\cdot 2\cdot 3l\cdot e^{-3l}(3l)^{3l}}{(l-1)!}\\
\leq \exp \bigg(-l\log l +l(\log l - \log(l-1))+l-1 + \frac{1}{2}\log(l-1) -\log\sqrt{2\pi}\\
+\log 18+\log l-3l+3l\log 3+3l\log l \bigg)\\
= \exp\left(2l\log l+l\left(3\log 3-2+ \log\frac{l}{l-1} + \frac{\log l}{l} + \frac{\log(l-1)}{2l}+\frac{\log \frac{18}{\sqrt{2\pi}}-1}{l}\right)\right)\\
\leq \exp\left(2l\log l+1.6791 l\right).
\end{multline*}
When $m=3$, we have $l\geq \left\lceil e^{3 (\log 3)^2} \right\rceil = 38$ and we need to divide $|B_{k,0}^*(1)|$ 
by the common factor to get the correct bound for the term $|B_{k,0}|$. Hence
\begin{multline*}
|B_{k,0}| \le \frac{e^{-0.215544\cdot 3l}}{(l-1)!}\left(6 + 8le^{-4l}(4l)^{4l}+2.55e^{-2\cdot 4l}(8l)^{4l}\right) \leq \frac{3e^{-0.215544\cdot 3l}}{(l-1)!}\cdot 8le^{-4l}(4l)^{4l} \\
\leq \exp \bigg(-l\log l +l(\log l - \log(l-1))+l-1 + \frac{1}{2}\log(l-1) -\log\sqrt{2\pi}\\
-3\cdot 0.215544 l + \log 24  +\log l - 4l +4l\log 4 + 4l\log l \bigg)\\
= \exp\Bigg(3l\log l+l\Bigg(4\log 4-3 -3\cdot 0.215544 + \log\frac{l}{l-1}+ \frac{\log l}{l} + \frac{\log(l-1)}{2l}+ \frac{\log \frac{24}{\sqrt{2\pi}} - 1}{l}\Bigg) \Bigg) \\
\leq \exp\left(3l\log l+2.1016 l\right).
\end{multline*}
When $m=4$, it holds $l\geq \left\lceil e^{4 (\log 4)^2} \right\rceil = 2181$ and again we have to divide $|B_{k,0}^*(1)|$ by the common factor to get the correct bound for the term $|B_{k,0}|$. Hence
\begin{multline*}
|B_{k,0}| \le \frac{e^{-0.173121\cdot4l}}{(l-1)!}\left(2\cdot 5le^{-5l}(5l)^{5l}+2.55e^{-2\cdot5l}(2\cdot5l)^{5l}+24\cdot3.632^{l-1}\right)\\
\leq \frac{3e^{-0.173121\cdot4l}}{(l-1)!}\cdot 10 le^{-5l}(5l)^{5l} \\
\leq \exp \bigg(-l\log l +l(\log l - \log(l-1))+l-1 + \frac{1}{2}\log(l-1) -\log\sqrt{2\pi}\\
-4 \cdot 0.173121 l + \log 3 + \log 10 + \log l -5l + 5l \log 5 + 5l \log l \bigg) \\
= \exp \Bigg( 4l \log l + l \Bigg( 5 \log 5 - 4 -4 \cdot 0.173121 + \log \frac{l}{l-1} + \frac{\log l}{l} + \frac{\log (l-1)}{2l} + \frac{\log \frac{30}{\sqrt{2 \pi}} -1}{l} \Bigg) \Bigg) \\
\leq \exp (-l \log l + 3.3612 l).
\end{multline*}
In all three cases, the coefficient of $l$ is a decreasing function in $l$.
\end{proof}

\begin{lemma}\label{Estimates2}
Let $l\geq e^{m (\log m)^2}$ when $m \ge 3$, and $l \ge e^e$ when $m=2$. We have
\begin{equation*}
\sum_{j=1}^{m}| L_{k,j} | \le \exp\left(-l\log l +l\left(\left(m+\frac{1}{2}\right)\log m-(\kappa_m+1) m-0.02394\right)\right)
\end{equation*}
for $j=1, \ldots ,m$, $k=0,1, \ldots ,m$, and $m\geq 5$. When $m=2,3,4$, we have the bounds
\begin{equation}\label{smallvalrem}
\begin{cases}
\sum_{j=1}^2 |L_{k,j}|\leq \exp\left(-l\log l+0.3654l\right),& \textrm{when m=2};\\
\sum_{j=1}^3 |L_{k,j}|\leq \exp\left(-l\log l+0.5139l\right),& \textrm{when m=3};\\
\sum_{j=1}^4 |L_{k,j}|\leq \exp\left(-l\log l+1.6016l\right),& \textrm{when m=4}.
\end{cases}
\end{equation}
\end{lemma}

\begin{proof}
We proceed as in the proof of Lemma \ref{Estimates}: first we bound the terms $L^*_{k,j}(1)$, then sum them, and finally factor out the common divisor to obtain the bound $R(l)$.
According to equation \eqref{Lkjtl-1}, we have the representation
\begin{equation*}
L^*_{k,j} (t)= \frac{t^{L+1}}{(l-1)!} \int_{0}^{j}e^{(j-x)t} (0-x)^l (1-x)^l \cdots (k-x)^{l-1} \cdots (m-x)^l \mathrm{d}x.
\end{equation*}
The expression $| x (1-x) \cdots (m-x) |$ attains its maximum in the interval $] 0, m [$  for the first time when $0 < x < 1$, so
\[
\max_{0 < x < m} | x (1-x) \cdots (m-x) | \le \frac{m!}{5!} \max_{0 < x < 1} |x(1-x)(2-x)(3-x)(4-x)(5-x)|= \frac{m!16.91}{120}.
\]
Thus we may estimate
\begin{equation*}
| L^*_{k,j}(1) | \le \;\frac{e^j}{(l-1)!} \int_0^j e^{-x}\frac{\prod_{r=0}^m |r-x|^{l}}{|k-x|}\mathrm{d}x
\le \;\frac{(m!)^l16.91^{l-1}}{120^{l-1}(l-1)!} \left(e^j - 1\right) 
\end{equation*}
when $m\geq 5$. Using the estimate $\sum_{j=1}^m (e^j-1) < e^m\frac{e}{e-1}$,  and summing together the terms $L^*_{j,k}(1)$, we get
\[
\sum_{j=1}^m |L^*_{j,k}(1)| < \frac{1.582e^m(m!)^l16.91^{l-1}}{120^{l-1}(l-1)!}.
\]
Again we divide by the common factor $D_{m,l}$. Thus the new values $L_{k,j}$ satisfy:
\begin{multline*}
\sum_{j=1}^m | L_{k,j} | < \frac{1.582(m!)^l16.91^{l-1}}{120^{l-1}(l-1)!} \cdot e^{m-\kappa_m ml}
< \exp\left(-\left(l-\frac{1}{2}\right)\log(l-1)+l-\log\sqrt{2\pi}\right)\\
\times \exp\left(l\left(\left(m+\frac{1}{2}\right)\log m-m+\log \sqrt{2\pi}+\frac{1}{12m}+\frac{m}{l}-\kappa_m m+\log16.91-\log 120\right)\right) \\
\times \exp\left(\log 120-\log 16.91+\log 1.582\right).
\end{multline*}
Now
\begin{equation*}\label{est3}
\frac{m}{l}+\log(l)-\log(l-1)+\frac{\log (l-1)}{2l} + 1 - \frac{\log (\sqrt{2 \pi})}{l} \le 1.00003,
\end{equation*}
\begin{equation*}\label{est4}
\log (\sqrt{2 \pi}) + \frac{1}{12m} + \log 16.91 - \log 120 \le -1.02398,
\end{equation*}
and
\begin{equation*}\label{est5}
\frac{\log 120 - \log 16.91 + \log 1.582}{l} \le 5.73802 \cdot 10^{-6}
\end{equation*}
because $m \ge 5$ and $l \ge e^{5 (\log 5 )^2}$. Together these estimates yield
\begin{equation*}
\sum_{j=1}^m | L_{k,j} | \le \exp\left(-l\log l +l\left(\left(m+\frac{1}{2}\right)\log m-(\kappa_m+1) m-0.02394 \right)\right).
\end{equation*}

When $m=2,3,4$,  we can bound the terms $L_{k,j}$ in the following way:
\[
\frac{(l-1)!}{e^j}|L_{k,j}^*|\leq \int_0^j e^{-x}\frac{\prod_{r=0}^m|r-x|^l}{|k-x|}\mathrm{d}x\leq
\begin{cases}
\frac{2^{l+1}}{3^{3(l-1)/2}},\ &\textrm{when }m=2; \\
6 ,\ &\textrm{when }m=3;\\
24 \cdot 3.632^{l-1}, \ &\textrm{when }m=4.
\end{cases}
\]
We may now move to estimating the sums $\sum_{j=1}^m |L_{k,j}|$ for small $m$. When $m=2$, we have $l \ge 16$ and
\begin{multline*}
\sum_{j=1}^2 |L_{k,j}| = \sum_{j=1}^2 |L_{k,j}^*| \le \frac{1}{(l-1)!}\cdot \frac{2^{l+1}}{3^{3(l-1)/2}}\left(e+e^2\right)\\
\leq \exp \bigg(-l\log l +l(\log l - \log(l-1)) +l -1 + \frac{1}{2}\log(l-1) -\log\sqrt{2\pi}+\log \left(e+e^2 \right) \\
+(l+1)\log 2-\frac{3(l-1)}{2}\log 3 \bigg) 
= \exp\Bigg(-l\log l +l\Bigg(\log 2-\frac{3}{2}\log 3 +1 + \log\frac{l}{l-1} \\ 
+ \frac{\log(l-1)}{2l}+\frac{\log(e+e^2) + \log \frac{2}{\sqrt{2 \pi}} + \frac{3}{2}\log 3 -1}{l} \Bigg)\Bigg) \leq \exp\left(-l\log l+0.3654 l\right).
\end{multline*}
When $m=3$, we have $l \ge 38$ and we need to divide to remove the common factors. Then
\begin{multline*}
\sum_{j=1}^3 |L_{k,j}| \le \frac{e^{-0.215544\cdot 3l}}{(l-1)!} \cdot 6 \left(e+e^2+e^3\right)\\
\leq \exp\bigg(-l\log l +l(\log l - \log(l-1)) +l -1 + \frac{1}{2}\log(l-1) -\log\sqrt{2\pi} \\
- 3\cdot 0.215544 l+\log 6 +\log \left(e+e^2+e^3 \right) \bigg) \\
\leq \exp\Bigg(-l\log l + l\Bigg(1-3\cdot 0.215544 + \log\frac{l}{l-1} + \frac{\log(l-1)}{2l}+\frac{\log(e+e^2+e^3)+ \log \frac{6}{\sqrt{2\pi}} -1}{l}\Bigg)\Bigg) \\
\leq \exp\left(-l\log l+0.5139 l\right).
\end{multline*}
When $m=4$, we have $l \ge 2181$ and again we divide by the common factor. Thus
\begin{multline*}
\sum_{j=1}^m |L_{k,j}| \le \frac{e^{-0.173121 \cdot 4l}}{(l-1)!}\cdot 24 \cdot 3.6^{l-1} \left(e+e^2+e^3+e^4\right) \\
\leq \exp\bigg(-l\log l +l(\log l - \log(l-1)) +l -1 + \frac{1}{2}\log(l-1) -\log\sqrt{2\pi} \\
-4 \cdot 0.173121 l + \log 24 + (l-1) \log 3.632 + \log \left(e+e^2+e^3+e^4\right) \bigg)\\
= \exp \Bigg( -l \log l + l \Bigg( 1-4 \cdot 0.173121 + \log 3.632 + \log \frac{l}{l-1} + \frac{\log (l-1)}{2l} \\
+ \frac{\log (e+e^2+e^3+e^4) + \log \frac{24}{3.632 \sqrt{2 \pi}} -1}{l} \Bigg) \Bigg) \leq \exp (-l \log l + 1.6016 l).
\end{multline*}
Again, in all three cases, the coefficient of $l$ is a decreasing function in $l$.
\end{proof}

\section{Measure}

We will apply Lemma \ref{lemma23}. The determinant condition \eqref{DET} is certainly satisfied by Lemma \ref{Lemma5.1}
and \eqref{numeerinenlineaari}.
According to Lemmas \ref{Estimates} and \ref{Estimates2}, we have $|B_{k,0}(l)|\le Q(l)=e^{q(l)}$ and $\sum_{j=1}^{m}| L_{k,j} | \le R(l)= e^{-r(l)}$, where
\begin{equation}\label{ql}
q(l)=ml\log l+l((m+1)\log (m+1)-(1+\kappa_m)m+0.0000525),
\end{equation}
\begin{equation}\label{-rl}
-r(l)= - l \log l + l \left(\left(m + \frac{1}{2} \right) \log m - (1+ \kappa_m)m-0.02394\right),
\end{equation}
for all $k,j=0,1, \ldots ,m$, $m \ge 5$.
Comparing formulas \eqref{ql} and \eqref{-rl} to \eqref{qn} and \eqref{-rn}, we have
\begin{equation*}
\begin{cases}
a=m,\\
b=(m+1)\log (m+1)-(1+ \kappa_m)m+\delta,\quad \delta = 0.0000525;\\
c=1,\\
d=\left(m+\frac{1}{2}\right)\log m-(1+ \kappa_m)m-0.02394.
\end{cases}
\end{equation*}
Now, with $s(m)=m(\log m)^2$, the formulas in \eqref{BCD} give
\begin{equation*}
\begin{cases}
B = b+\frac{ad}{c} \\
\phantom{B}= m^2\log m-(1+ \kappa_m)m^2+(m+1)\log (m+1)+ \frac{1}{2}m\log m-(1.02394+ \kappa_m)m+\delta,\\
C =a=m, \\
D =a+b+ae^{-s(m)} = (m+1)\log (m+1)-\kappa_m m+\delta+\frac{m}{e^{m(\log m)^2}}
\end{cases}
\end{equation*}
for all $m \ge 5$. Recall also the shorthand notations $u= 1+ \frac{\log (s(m))}{s(m)}$ and $v = 1 - \frac{d}{s(m)}$.

For the small values $m=2,3,4$, we compare equations \eqref{smallvalpol} and \eqref{smallvalrem} to \eqref{qn} and \eqref{-rn}. Again $a=m$ and $c=1$, and moreover
\begin{equation}\label{bdsmallm}
\begin{cases}
b=1.6791, \\
d=0.3654
\end{cases}
\text{for }m=2; \quad
\begin{cases}
b=2.1016,\\
d=0.5139
\end{cases}
\text{for }m=3; \quad
\begin{cases}
b=3.3612,\\
d=1.6016
\end{cases}
\text{for }m=4.
\end{equation}
Hence, with $s(2)=e$ and $s(m)=m(\log m)^2$ for $m=3,4$, we get
\begin{equation}\label{BCDsmallm}
\begin{cases}
B=2.4099\\
C=2, \\
D=3.8111
\end{cases}
\text{for }m=2; \quad
\begin{cases}
B=3.6433,\\
C=3,\\
D=5.1819
\end{cases}
\text{for }m=3; \quad
\begin{cases}
B=9.7676,\\
C=4, \\
D=7.3631
\end{cases}
\text{for }m=4.
\end{equation}

We may thus finally establish our Main result \ref{mainresult}:

\begin{proof}[Proofs of Theorem \ref{mainresult} and Corollary \ref{corollaryofmainres}]
The values above have been achieved with the choice $\Theta_j = e^j$. 
Combining them with Lemma \ref{lemma23} leads straight to the result \eqref{mainreseq}. 
Corollary \ref{corollaryofmainres} follows likewise by plugging these values into Corollary \ref{corollaari24}.
\end{proof}

Estimate \eqref{mainestimate} still requires a bit more work.

\begin{proof}[Proof of Theorem \ref{MAINCOROLLARY}]
Let us first consider the case with $m\geq 5$. 
According to Lemma \ref{lemma23}, we have
$$
1 < | \Lambda | 2 (2H)^{\frac{a}{c}} e^{\epsilon (H) \log (2H) +D} =   |\Lambda| H^{\frac{a}{c}+Y}=   |\Lambda| H^{m+Y},
$$
where
\begin{equation}\label{Ylauseke} 
\begin{split}
Y := \;&  \frac{1}{\log H}\left(B z\!\left(\frac{\log (2H)}{1 - \frac{d}{s(m)}}\right) + 
m \log\left(z\!\left(\frac{\log (2H)}{1 - \frac{d}{s(m)}}\right)\right) +D+ \left(m+1 \right)\log 2\right)   \\
 \le \;&  \frac{1}{\log H}\left( \frac{uB}{v}\frac{\log(2H)}{\log\log(2H)} + 
m\log\left( \frac{u}{v}\frac{\log(2H)}{\log\log(2H)}\right) + D+ \left(m+1 \right)\log 2\right)   \\
\end{split}
\end{equation}
by estimate \eqref{nhattuylaraja}.
By recalling our assumption $\log H\ge s(m)e^{s(m)}$,
it is obvious from the expression \eqref{Ylauseke} that the terms corresponding to the parameters $C$ and $D$ contribute much less 
than the term corresponding to the parameter $B$. The first task is to bound them in such a way that they only slightly increase 
the constant term in the expression for the parameter $B$. 
Let us start with the terms $D$ and $(m+1)\log 2$. 
We have
\begin{equation*}
\begin{split}
D+(m+1)\log 2 &=(m+1)\log (m+1)-\kappa_m m+\delta+\frac{m}{e^{m(\log m)^2}}+(m+1)\log 2\\
&=(m+1)\log(m+1) +m\left(\frac{\delta}{m} +\frac{1}{e^{m(\log m)^2}}+\log 2+\frac{\log 2}{m}-\kappa_m\right)\\
&\leq (m+1)\log(m+1)+\frac{1}{2}m.
\end{split}
\end{equation*}

Since $v\log\log(2H)\geq 1$, we may estimate
\[
m\log\left(\frac{u\log(2H)}{v\log\log (2H)}\right)\leq m\log(u\log(2H)).
\]
Hence, the estimate becomes
\begin{equation*}
\begin{split}
Y &\le  \frac{1}{\log H}\left( \frac{uB}{v}\frac{\log(2H)}{\log\log(2H)} + m\log\left(u\log(2H)\right) + (m+1)\log(m+1)+\frac{1}{2}m \right)\\
&\leq  \frac{1}{\log H}\left( \frac{uB}{v}\frac{\log(2H)}{\log\log(2H)} + m\log\left(2\log H\right) + (m+1)\log(m+1)+\frac{1}{2}m \right)\\
&\leq  \frac{1}{\log H}\left( \frac{uB}{v}\frac{\log(2H)}{\log\log(2H)} + \frac{5}{4}m\log\left(2\log H\right)\right).
\end{split}
\end{equation*}
We have now derived
\begin{equation*}
\begin{split}
Y &\leq \frac{1}{\log H}\left( \frac{uB}{v}\frac{\log(2H)}{\log\log(2H)} + \frac{5}{4}m\log\left(2\log H\right)\right)\\
&\leq \frac{1}{\log\log H}\left(\frac{\log (2H)}{\log H}\cdot \frac{uB}{v}+\frac{5}{4}\cdot \frac{m(\log\log  H) \log(2\log H)}{\log H}\right)\\
&=\frac{u}{v\log\log H}\left(B+\frac{1}{\log H}\left(\log(2)B+\frac{5vm(\log\log H)\log(2\log H)}{4u}\right)\right).
\end{split}
\end{equation*}

When $m=5$, the above formulation gives
\[
Y \le \frac{u}{v\log\log H}\left(B+0.0002069\right).
\]

When $m\geq 6$, we proceed as follows. Notice now that roughly estimating we have $B\leq m^2\log m-\kappa_mm^2$ because $\log(m+1)-(1.02394+\kappa_m)m+\delta\leq 0$ and $-m^2+m\log(m+1)+\frac{1}{2}m\log m\leq 0$. Furthermore, $\kappa_m\geq 0.32$, when $m\geq 6$. Since $0<v\leq 1\leq u$, we have now derived the inequality
\begin{equation*}
\begin{split}
Y &\leq \frac{u}{v\log\log H}\left(B+\frac{m^2}{\log H}\left(\log (2) \log m-\kappa_m\log(2)+\frac{5(\log\log H)(\log(2+\log H))}{4m} \right)\right)\\
&\leq \frac{u}{v\log\log H} \left(B+10^{-6} \right).
\end{split}
\end{equation*}

When $m\geq 6$, let us take a closer look at 
\begin{multline*}
f(m):=\frac{u(B+10^{-6})}{v m^2\log m} = \\
\frac{\left(1 + \frac{1}{m\log m}+\frac{2\log\log m}{m(\log m)^2} \right) 
\left( 1 -\frac{1+\kappa_m}{\log m} + \frac{(m+1)\log (m+1)}{m^2 \log m} +\frac{1}{2m} - 
\frac{1.02394+\kappa_m}{m \log m} + \frac{0.0000525+10^{-6}}{m^2 \log m} \right)}
{1 - \frac{1}{\log m} - \frac{1}{2m\log m} + \frac{1+\kappa_m}{ (\log m)^2}}.
\end{multline*}
When $m=5$, define the value $f(5)$ using the same formula but $0.0002069$ in the place of of $10^{-6}$. Before moving any further, notice that the value of the expression $f(m)$ can be estimated, 
and compared against the value of $\left(1-\frac{\kappa_m}{\log m}\right)\left(1-\frac{2\kappa_m}{(\log m)^2}\right)$ 
when $5\leq m\leq 14$. The calculations are performed by Sage \cite{Sage}. The values of both functions are presented in the following table:
\[
\begin{array}{|c|c|c|}
\hline m & f(m) & \left(1-\frac{\kappa_m}{\log m}\right)\left(1-\frac{2\kappa_m}{(\log m)^2}\right)\\
\hline 
5 & 0.4638\ldots & 0.5324\ldots\\
6 & 0.6159\ldots & 0.6551\ldots\\
7 &  0.6032\ldots & 0.6469\ldots\\
8 & 0.6158\ldots & 0.6603\ldots\\
9 & 0.5768\ldots & 0.6296\ldots\\
10 & 0.6366\ldots& 0.6831\ldots\\
11 & 0.5995\ldots & 0.6529\ldots\\
12 & 0.6444\ldots & 0.6936\ldots\\
13 &0.6286\ldots & 0.6812\ldots\\
14 & 0.6203\ldots & 0.6749\ldots \\
\hline
\end{array}
\]
It is evident from these values that 
$f(m) \leq \left(1-\frac{\kappa_m}{\log m}\right)\left(1-\frac{2\kappa_m}{(\log m)^2}\right)$ 
when $5\leq m\leq 14$. Actually, when $m\ne 6$, the coefficient $2$ could be replaced by the better coeffient $2.5$.

We have thus shown $f(m) \leq \left(1-\frac{2\kappa_m}{(\log m)^2}\right)\left(1-\frac{\kappa_m}{\log m}\right)$ when $5\leq m\leq 14$, and the proof is ready for $5\leq m\leq 14$. For the rest of the proof we assume that $m\geq 15$ meaning also that $0.5\leq \kappa_m\leq 0.756$. Let us continue by writing
\[
f(m)=g(m)h(m),
\]
where
\[
g(m):=\frac{1-\frac{1}{\log m}}{1-\frac{1}{\log m}-\frac{1}{2m\log m}+\frac{1+\kappa_m}{(\log m)^2}}=
1-\frac{1+\kappa_m}{(\log m)^2}\frac{1-\frac{\log m}{2m(1+\kappa_m)}}{1-\frac{1}{\log m}-\frac{1}{2m\log m}+\frac{1+\kappa_m}{(\log m)^2}}
\]
and
\[ 
h(m):=
\frac{\left(1 + \frac{1}{m\log m}+\frac{2\log\log m}{m(\log m)^2} \right)
\left( 1 -\frac{1+\kappa_m}{\log m} + \frac{(m+1)\log (m+1)}{m^2 \log m} +\frac{1}{2m} - 
\frac{1.02394+\kappa_m}{m \log m} + \frac{0.0000535}{m^2 \log m} \right)}
{1 - \frac{1}{\log m}}.
\]

First we show that $g(m)\leq 1-\frac{1+\kappa_m}{(\log m)^2}$. This claim is equivalent to
\[
2+\frac{1}{m}-\frac{2(1+\kappa_m)}{\log m}-\frac{(\log m)^2}{(1+\kappa_m)m}\geq 0, 
\]
which is true when $m\geq 15$ because $\frac{2(1+\kappa_m)}{\log m}<\frac{4}{\log 15}<1.48$ and $\frac{(\log m)^2}{(1+\kappa_m)m}<\frac{(\log m)^2}{m}<0.49$.  We still need to  prove that 
\begin{equation*}\label{hmarvio}
h(m)\leq 1-\frac{\kappa_m}{\log m}.
\end{equation*}
Let us now look at the second term in the numerator of $h(m)$.
First take a look at the ratio
\[
\frac{(m+1)\log (m+1)}{m^2 \log m}\leq \frac{1}{m}+\frac{1}{m^2\log m}+\frac{1}{m^2}+\frac{1}{m^3\log m}.
\]

We have
\begin{multline*}
1 -\frac{1+\kappa_m}{\log m} + \frac{(m+1)\log (m+1)}{m^2 \log m} +\frac{1}{2m} - \frac{1.02394+\kappa_m}{m \log m} + \frac{0.0000535}{m^2 \log m}\\
\leq 1 -\frac{1+\kappa_m}{\log m} + \frac{1}{m}+\frac{1}{m^2\log m}+\frac{1}{m^2}+\frac{1}{m^3\log m}+\frac{1}{2m} - \frac{1.02394+\kappa_m}{m \log m} + \frac{0.0000535}{m^2 \log m}\\
<1-\frac{1+\kappa_m}{\log m}+\frac{3}{2m},
\end{multline*}
since
\[
\frac{1}{m^2}+\frac{1.0000535}{m^2\log m}+\frac{1}{m^3\log m}-\frac{1.02394+\kappa_m}{m\log m}<0.
\]
Thus, we have
\begin{multline*}
\frac{\left(1 + \frac{1}{m\log m}+\frac{2\log\log m}{m(\log m)^2} \right) \left( 1 -\frac{1+\kappa_m}{\log m} + \frac{(m+1)\log (m+1)}{m^2 \log m} +\frac{1}{2m} - \frac{1.02394+\kappa_m}{m \log m} + \frac{0.0000535}{m^2 \log m} \right)}{1 - \frac{1}{\log m} - \frac{1}{2m\log m} + \frac{1+\kappa_m}{ (\log m)^2}}\\
<
\left(1+\frac{1}{m\log m}+\frac{2\log \log m}{m(\log m)^2}\right)\left(1-\frac{\frac{\kappa_m}{\log m}-\frac{3}{2m}}{1-\frac{1}{\log m}}\right).
\end{multline*}

Let us now prove that
\[
\frac{\left(1 + \frac{1}{m\log m}+\frac{2\log\log m}{m(\log m)^2} \right)\left(1-\frac{1+\kappa_m}{\log m}+\frac{3}{2m}\right)}{1-\frac{1}{\log m}}<1-\frac{\kappa_m}{\log m}.
\]
This is done by showing that
\[
1-\frac{1+\kappa_m}{\log m}+\frac{3}{2m}<\left(1-\frac{\kappa_m}{\log m}\right)\left(1-\frac{1}{\log m}\right)\left(1-\frac{1}{m\log m}-\frac{2\log\log m}{m(\log m)^2}\right),
\]
because then
\begin{multline*}
\frac{\left(1 + \frac{1}{m\log m}+\frac{2\log\log m}{m(\log m)^2} \right)\left(1-\frac{1+\kappa_m}{\log m}+\frac{3}{2m}\right)}{1-\frac{1}{\log m}}<\\
\frac{\left(1 + \frac{1}{m\log m}+\frac{2\log\log m}{m(\log m)^2} \right)\left(1-\frac{\kappa_m}{\log m}\right)\left(1-\frac{1}{\log m}\right)\left(1-\frac{1}{m\log m}-\frac{2\log\log m}{m(\log m)^2}\right)}{1-\frac{1}{\log m}}
<1-\frac{\kappa_m}{\log m}.
\end{multline*}
Notice first that
\begin{multline*}
\left(1-\frac{\kappa_m}{\log m}\right)\left(1-\frac{1}{\log m}\right)\left(1-\frac{1}{m\log m}-\frac{2\log\log m}{m(\log m)^2}\right)\\
>1-\frac{1+\kappa_m}{\log m}+\frac{\kappa_m}{(\log m)^2}-\frac{1}{m\log m}+\frac{1+\kappa_m}{m(\log m)^2}-\frac{2\log\log m}{m(\log m)^2}\\
>1-\frac{1+\kappa_m}{\log m}+\frac{\kappa_m}{(\log m)^2}-\frac{2}{m\log m}+\frac{1+\kappa_m}{m(\log m)^2},
\end{multline*}
so we have to show that
\[
\frac{3}{2m}<\frac{\kappa_m}{(\log m)^2}-\frac{2}{m\log m}+\frac{1+\kappa_m}{m(\log m)^2}.
\]
This is equivalent to $3(\log m)^2+4\log m<2\kappa_m m+2+2\kappa_m$. When $m\geq 15$, the right hand side of the inequality  is at least $2m+3$, since $\kappa_m\geq 0.5$. The inequality
\[
3(\log m)^2+4\log m<2m+3
\]
is true when $m\geq 14.74$, and hence for all integer values $m\geq 15$. The proof is complete for $m \ge 5$.

Let us now move to the small values of $m$. We use estimate \eqref{Ylauseke} with the values in \eqref{bdsmallm} and \eqref{BCDsmallm}. When $m=2$, we have $\log H\ge s(2)e^{s(2)} = e^{e+1}$, $\frac{u}{v} \le 1.5804$, and hence
\begin{multline*}
Y\leq \frac{1}{\log H}\left(\frac{uB}{v}\frac{\log (2H)}{\log\log (2H)}+2\log\left(\frac{u}{v}\frac{\log(2H)}{\log\log(2H)}\right)+D+3\log 2\right)\\
\le \frac{1}{\log\log H}\bigg(1.5804\cdot 2.4099\frac{\log (2H)}{\log H}+2\cdot 0.7732\frac{(\log\log H)^2}{\log H}+\frac{(3.8111+3\log 2)\log\log H}{\log H}\bigg)\\
\leq \frac{4.93}{\log \log H}.
\end{multline*}
When $m=3$, we have $\log H\ge s(3)e^{s(3)} = 3(\log 3)^2 e^{3 (\log 3)^2}$, $\frac{u}{v} \le 1.5796$, and hence
\begin{multline*}
Y\leq \frac{1}{\log H}\left(\frac{uB}{v}\frac{\log (2H)}{\log\log (2H)}+3\log\left(\frac{u}{v}\frac{\log(2H)}{\log\log(2H)}\right)+D+4\log 2\right)\\
\leq  \frac{1}{\log\log H}\bigg(1.5796\cdot 3.6433 \frac{\log(2H)}{\log H}+ 3\cdot 0.7699\frac{(\log\log H)^2}{\log H}+\frac{(5.1819+4\log 2)\log\log H}{\log H}\bigg)\\
\leq \frac{6.49}{\log\log H}.
\end{multline*}
When $m=4$, we have $\log H\ge s(4)e^{s(4)} = 4(\log 4)^2 e^{4 (\log 4)^2}$, $\frac{u}{v} \le 1.5984$, and hence
\begin{multline*}
Y\leq \frac{1}{\log H}\left(\frac{uB}{v}\frac{\log (2H)}{\log\log (2H)}+4\log\left(\frac{u}{v}\frac{\log(2H)}{\log\log(2H)}\right)+D+5\log 2\right)\\
\leq  \frac{1}{\log\log H}\bigg(1.5984 \cdot 9.7676 \frac{\log (2H)}{\log H}+4\cdot 0.8144\frac{(\log\log H)^2}{\log H}+\frac{(7.3631+5\log 2)\log\log H}{\log H}\bigg)\\
\leq \frac{15.7}{\log\log H}.
\end{multline*}
\end{proof}

\section{Sparse polynomials}

The method presented in this paper suits very well for obtaining bounds for sparse polynomials of $e$, namely, 
polynomials which have a considerable number of coefficients equal to zero.
Let the pairwise different non-negative integers $\beta_0=0,\beta_1,\dots,\beta_{m_1}$ be the exponents of the sparse polynomial
$P(x)=\lambda_0+\lambda_1 x^{\beta_1}+\ldots +\lambda_{m_1}x^{\beta_{m_1}}\in\mathbb{Z}_{\mathbb{I}}[x]$.

\begin{theorem}
Let $P(x)=\lambda_0+\lambda_1 x^{\beta_1}+\ldots +\lambda_{m_1} x^{\beta_{m_1}}$ 
be a polynomial with at most $m_1+1\geq 2$ non-zero coefficients, 
and of degree $m_2\geq 4$, where $m_2\geq m_1+1$. Suppose $\log H\geq m_2(\log m_2)^2e^{m_2(\log m_2)^2}$.
Then the bound
\[
|P(e)|>H^{-m_1-\frac{\rho(m_1^2+3m_1+2)\log m_2}{\log \log H}}
\]
holds for all 
$\,\overline{\lambda}=(\lambda_0, \lambda_1, \ldots ,\lambda_m)^T \in \mathbb{Z}_{\mathbb{I}}^{m+1}\setminus\{\overline{0}\}$ 
with $\max_{1\le i\le m} \{|\lambda_i|\} \le H$,
where the constant $\rho\leq 12.88$ for all $m_2\geq 4$, and $\rho\leq 2$ when $m_2\geq 11$.
\end{theorem}

\begin{proof} This boils down to estimating the size of the terms $Q(n)$ and $R(n)$. 
We use the polynomial expression $\Omega(w,\overline{\beta})$. Now the polynomial in question is 
$\prod_{j=0}^{m_1}(\beta_j-w)^{l_j}$, where $\beta_j$ are the exponents of the polynomial, so $0\leq \beta_j\leq m_2$ for all $j$. 
Furthermore, we know that $l_j=l$ with the exception of one index, in which case it is $l-1$. 
We may assume that the index in question is $k$, namely, that the terms $B_{k,0}$, $B_{k,j}$ and $L_{k,j}$ 
correspond to the polynomials with $l_k=l-1$. Furthermore, we assume $l\geq s(m_2)e^{s(m_2)}$.

Let us now estimate the size of the polynomial using the same method as earlier. We have
\[
B^*_{k,0}(t)=\frac{1}{(l-1)!}\int_0^{\infty}e^{-xt}\frac{\prod_{j=0}^{m_1}(\beta_j-x)^l}{\beta_k-x}\mathrm{d}x,
\]
and we need the value at $t=1$. 
First the integral needs to be split into integrals over the intervals $[0,m_2]$, $[m_2,2m_2l]$ and $[2m_2l,\infty)$. 
Let us start by looking at the first integral. We have
\[
\frac{1}{(l-1)!}\int_0^{m_2}e^{-xt}\frac{\prod_{j=0}^{m_1}|\beta_j-x|^l}{|\beta_k-x|}\mathrm{d}x
\leq 
\frac{1}{(l-1)!}\int_0^{m_2}m_2^{l(m_1+1)-1} \mathrm{d}x=\frac{m_2^{(m_1+1)l}}{(l-1)!}.
\]

Next we estimate the integral on the interval $[m_2,2m_2l]$. Now
\[
\frac{\prod_{j=0}^{m_1}|\beta_j-x|^l}{|\beta_k-x|}\mathrm{d}x\leq x^{(m_1+1)l-1}.
\]
Let us now look at the function $f(x)=e^{-x}x^{(m_1+1)l-1}$. We have
\[
f'(x)=-e^{-x} x^{(m_1+1)l-1}+((m_1+1)l-1)e^{-x}x^{(m_1+1)l-2}=0,
\]
when $x_0=(m_1+1)l-1$. Hence, the integral can be estimated to be
\[
\frac{1}{(l-1)!}\int_{m_2}^{2m_2l}e^{-x}\frac{\prod_{j=0}^{m_1}|\beta_j-x|^l}{|\beta_k-x|}\mathrm{d}x\leq \frac{2m_2l e^{-(m_1+1)l+1}((m_1+1)l-1)^{(m_1+1)l-1}}{(l-1)!}.
\]

Finally, let us estimate the third integral
\[
\frac{1}{(l-1)!}\int_{2m_2l}^{\infty}e^{-x}\frac{\prod_{j=0}^{m_1}|\beta_j-x|^l}{|\beta_k-x|}\mathrm{d}x\leq 
\frac{1}{(l-1)!}\int_{2m_2l}^{\infty}e^{-x} x^{(m_1+1)l-1}\mathrm{d}x.
\]
Again, we use the function $f(x)=e^{-x} x^{(m_1+1)l-1}$. Since this function obtains its maximum at $x_0=(m_1+1)l-1$, it is decreasing when $x>x_0$. We also have $2m_2l\geq (m_1+1)l \ge x_0$. Hence, we may estimate
\[
\frac{1}{(l-1)!}\int_{2m_2l}^{\infty}e^{-x} x^{(m_1+1)l-1}\mathrm{d}x\leq \frac{1}{(l-1)!}\sum_{h=0}^{\infty} e^{-2m_2l-h}(2m_2l+h)^{(m_1+1)l-1}.
\]
Let us estimate the ratio between consecutive terms:
\[
\frac{e^{-2m_2l-h-1}(2m_2l+h+1)^{(m_1+1)l-1}}{e^{-2m_2l-h}(2m_2l+h)^{(m_1+1)l-1}}=e^{-1}\left(1+\frac{1}{2m_2l+h}\right)^{(m_1+1)l-1}\leq e^{-1/2}.
\]
The third integral can thus be estimated as a geometric sum:
\[
\frac{1}{(l-1)!}\int_{2m_2l}^{\infty}e^{-x} x^{(m_1+1)l-1}\mathrm{d}x\leq \frac{e^{-2m_2l}(2m_2l)^{(m_1+1)l-1}}{(l-1)!(1-e^{-1/2})}.\]
Hence,
\begin{multline*}
\frac{1}{(l-1)!}\int_0^{\infty}e^{-xt}\frac{\prod_{j=0}^{m_1}|\beta_j-x|^l}{|\beta_k-x|}\mathrm{d}x \le \\ 
\frac{m_2^{(m_1+1)l}}{(l-1)!}+\frac{1}{(l-1)!}2m_2le^{-(m_1+1)l+1}((m_1+1)l-1)^{(m_1+1)l-1}+\frac{e^{-2m_2l}(2m_2l)^{(m_1+1)l-1}}{(l-1)!(1-e^{-1/2})}.
\end{multline*}
Since $m_2\leq \frac{l(m_1+1)}{e}$, we have
\[
\frac{m_2^{(m_1+1)l}}{(l-1)!}<\frac{1}{(l-1)!}2m_2le^{-(m_1+1)l+1} ((m_1+1)l-1)^{(m_1+1)l-1},
\]
and since the function $f(x)$ peaks at $(m_1+1)l-1$, we have
\[
\frac{1}{(l-1)!}2m_2le^{-(m_1+1)l+1}((m_1+1)l-1)^{(m_1+1)l-1}
>\frac{e^{-2m_2l}(2m_2l)^{(m_1+1)l-1}}{(l-1)!(1-e^{-1/2})}.
\]
Therefore,
\begin{equation*}
\begin{split}
\frac{1}{(l-1)!}\int_0^{\infty}e^{-xt}\frac{\prod_{j=0}^{m_1}|\beta_j-x|^l}{|\beta_k-x|}\mathrm{d}x &\leq 3\frac{1}{(l-1)!}2m_2le^{-(m_1+1)l+1}((m_1+1)l-1)^{(m_1+1)l-1}\\
&\leq  \frac{2m_2le^{-(m_1+1)l}((m_1+1)l)^{(m_1+1)l}}{(l-1)!}.
\end{split}
\end{equation*}

We need to write the estimate as an exponential function. Using \eqref{(l-1)fact} we get
\begin{multline*}
\frac{2m_2le^{-(m_1+1)l}((m_1+1)l)^{(m_1+1)l}}{(l-1)!} \\
\leq 
\exp\bigg(m_1l\log l+l\log\frac{l}{l-1}+l(m_1+1)\log(m_1+1)-m_1l+\log l \\
+\frac{1}{2}\log (l-1)+\log m_2-1+\log 2-\frac{1}{2}\log (2\pi)\bigg).
\end{multline*}
Since $l\log \frac{l}{l-1}\leq 1$ and
\[
\frac{1}{l} \left( \log l+\frac{1}{2}\log (l-1)+\log m_2+\log 2-\frac{1}{2}\log (2\pi) \right) \leq 0.006<\log (m_1+1),
\]
we have
\[
\frac{1}{(l-1)!}\int_0^{\infty}e^{-xt}\frac{\prod_{j=0}^{m_1}|\beta_j-x|^l}{|\beta_k-x|}\mathrm{d}x \leq \exp\left(m_1l\log l+l((m_1+2)\log(m_1+1)-m_1)\right).\]

Let us now estimate the terms $L_{k,j}$. They have the following integral representations:
\[
\left| L_{k,j} \right|
= \left| \frac{1}{(l-1)!}\int_{0}^{\beta_j}e^{\beta_j-x}\frac{\prod_{i=0}^{m_1}(\beta_i-x)^l}{\beta_j-x}\mathrm{d}x \right|
\leq \frac{e^{\beta_j}m_2^{l(m_1+1)-1}}{(l-1)!}\int_0^{\beta_j}e^{-x}\mathrm{d}x
\leq \frac{e^{\beta_j}m_2^{l(m_1+1)-1}}{(l-1)!}.
\]
We obtain
\[
\sum_{j=1}^{m_1} \left| L_{k,j} \right| \leq \sum_{j=1}^{m_1}\frac{e^{\beta_j}m_2^{l(m_1+1)-1}}{(l-1)!}\leq 
\frac{e^{m_2+1}m_2^{l(m_1+1)-1}}{(l-1)!}.
\]
Now we need to write this as an exponential function:
\begin{multline*}
\sum_{j=1}^{m_1} \left| L_{k,j} \right| \leq \sum_{j=1}^{m_1}\frac{e^{\beta_j}m_2^{l(m_1+1)-1}}{(l-1)!}\leq 
\frac{e^{m_2+1}m_2^{l(m_1+1)-1}}{(l-1)!}\\
\leq \exp\left(m_2+1+(l(m_1+1)-1)\log m_2-\left(l-\frac{1}{2}\right)\log(l-1)+l-1-\frac{1}{2}\log(2\pi)\right)\\
\leq \exp\left(-l\log l+l(m_1+2)\log m_2\right).
\end{multline*}
Now
\[
\begin{cases}
e^{q(l)}\leq \exp\left(m_1l\log l+l((m_1+2)\log(m_1+1)-m_1)\right),\\
e^{-r(l)}\leq \exp\left(-l\log l+l(m_1+2)\log m_2\right).
\end{cases}
\] 
Comparing the above to \eqref{qn} and \eqref{-rn}, we get
\[
a=m_1,\quad
b=(m_1+2)\log(m_1+1)-m_1,\quad
c=1,\quad \textrm{and}\quad
d=(m_1+2)\log m_2,
\]
and by \eqref{BCD},
\[
\begin{cases}
B=b+\frac{ad}{c}\leq (m_1^2+3m_1+2)\log m_2-m_1,\\
C=a=m_1,\\
D=a+b+a e^{-m_1(\log m_1)^2}\\
\phantom{D}=(m_1+2)\log (m_1+1)+m_1 e^{-m_1(\log m_1)^2}\\
\phantom{D}\leq 2(m_1+2)\log(m_1+1).
\end{cases}
\]

Next we sum together the terms arising from the terms $C$ and $D$. We may estimate (see \eqref{Ylauseke})
\begin{multline*}
C \log\left(z\!\left(\frac{\log (2H)}{1 - \frac{d}{s(m)}}\right)\right) +D+ \left(m+1 \right)\log 2 \\
\leq  m_1 \log\left(\frac{u}{v}\cdot \frac{\log (2H)}{\log\log (2H)}\right)+2(m_1+2)\log(m_1+1)+(m_1+1)\log 2\\
\leq 3(m_1+2)\log\left(13\frac{\log (2H)}{\log\log(2H)}\right)
\end{multline*}
since $(m_1+1)\log 2\leq (m_1+2)\log (m_1+1)$, $\frac{u}{v}\leq 13$, and
\[
3(m_1+2)\log (m_1+1)\leq 2(m_1+2)\log\left(13\frac{\log (2H)}{\log\log(2H)}\right).
\]
Now we can combine this term with the term coming from the term $B$:
\begin{multline*}
\frac{1}{\log H}\left(Bn_2+C\log n_2+D+ \left(m_1+1 \right)\log 2\right) \\
\leq \frac{1}{\log H}\left( \frac{u((m_1^2+3m_1+2)\log m_2-m_1)\log(2H)}{v\log\log(2H)}+3(m_1+2)\log\left(13\frac{\log (2H)}{\log\log(2H)}\right)\right).\end{multline*}

Let us start by eliminating the last term with the term $-m_1\frac{u\log (2H)}{v\log\log(2H)}$. Notice that
\[
3(m_1+2)\log \left(13\frac{\log(2H)}{\log\log (2H)}\right)<3(m_1+2)\log \left(\frac{13}{9}\log(2H)\right)< 6(m_1+2)\log\log (2H).
\]
Hence, it suffices to show that $\frac{\log (2H)}{(\log \log (2H))^2}\geq 6\frac{(m_1+2)}{m_1}$. This is easy to do. Notice first that the function $f(x)=\frac{x}{(\log x)^2}$ is increasing when $x\geq e^{2}$, so we may estimate
\[
\frac{\log(2H)}{(\log\log (2H))^2}\geq \frac{\log(H)}{(\log\log (H))^2}\geq \frac{ 2\cdot 4(\log(4))^2e^
{4(\log(4))^2}}{(\log (2\cdot 4(\log(4))^2e^{4(\log(4))^2}))^2} >308,\]
while $6\frac{(m_1+2)}{m_1}\leq 18$.
Thus
\begin{multline*}
\frac{1}{\log H}\left( \frac{u((m_1^2+3m_1+2)\log m_2-m_1)\log (2H)}{v\log\log (2H)}+3(m_1+2)\log\left( 13\frac{\log(2H)}{\log\log(2H)}\right)\right)\\
\leq \frac{1}{\log \log H}\left(1+\frac{\log 2}{\log H}\right)\frac{u}{v} \left(m_1^2+3m_1+2 \right)\log m_2.
\end{multline*}
Finally, $\left(1+\frac{\log 2}{\log H}\right)\frac{u}{v}$ is always at most $12.88$ (the biggest value for $m_2=4$) and it is decreasing. When $m_2\geq 11$, the value of this expression is at most $2$. Computations are performed by Sage \cite{Sage}.
\end{proof}

As a corollary of the bound obtained for sparse polynomials, we get the following transcendence measure for an arbitrary integer power of $e$:

\begin{corollary} Assume $d \in \Z_{\ge 2}$ and $\log H\geq dm(\log(dm))^2e^{dm(\log (dm))^2}$. Then the bound
\[
\left|\lambda_0+\lambda_1e^d+\lambda_2e^{2d}+\ldots +\lambda_me^{md} \right|>\frac{1}{H^{\omega(m,H)}},
\]
where
\[
\omega(m,H)<m+\frac{\rho(m^2+3m+2)\log(dm)}{\log\log H},
\]
holds for all 
$\,\overline{\lambda}=(\lambda_0, \lambda_1, \ldots ,\lambda_m)^T \in \mathbb{Z}_{\mathbb{I}}^{m+1}\setminus\{\overline{0}\}$ 
with $\max_{1\le i\le m} \{|\lambda_i|\} \le H$, and $\rho$ as in the previous theorem.
\end{corollary}
\begin{proof}
Notice that now $m_1=m$ and $m_2=dm$. Substituting these values into the previous theorem immediately yields the result.
\end{proof}

\section*{Acknowledgements}
We are indebted to the anonymous referees for their critical reading and helpful suggestions.


\end{document}